\documentclass[final,onefignum,onetabnum]{siamonline171218}
\usepackage[T1]{fontenc}
\usepackage{booktabs}
\usepackage{amssymb}
\usepackage{algorithmic} 
\usepackage{algorithm} 
\usepackage{graphics}
\usepackage{graphicx}
\usepackage{tikz}
\usepackage{caption}
\usepackage{amsfonts}
\usepackage{mathabx}
\usepackage{xr-hyper} 
\usepackage{cleveref}

\crefname{subsection}{section}{sections}

\usepackage[normalem]{ulem}

\newsiamremark{remark}{Remark}
\newsiamremark{assumption}{Assumption}
\newsiamremark{example}{Example}
\newsiamremark{hypothesis}{Hypothesis}
\crefname{hypothesis}{Hypothesis}{Hypotheses}
\newsiamthm{claim}{Claim}

\usepackage{enumitem}
\setlist[enumerate]{leftmargin=.5in}
\setlist[itemize]{leftmargin=.5in}

\usepackage{geometry}

\def\argmin{ \mathop{{\rm argmin}}}

\newcommand{\co}{\mathrm{conv}\,}
\DeclareMathOperator{\conv}{conv}
\newcommand{\cl}{\mathrm{cl}\,}

\DeclareMathOperator{\dom}{dom}
\DeclareMathOperator{\epi}{epi}

\DeclareMathOperator{\gph}{gph}

\DeclareMathOperator{\ri}{ri}

\DeclareMathOperator{\para}{par}
\DeclareMathOperator{\rge}{rge}

\DeclareMathOperator{\sgn}{sgn}

\def\Limsup{\mathop{{\rm Lim}\,{\rm sup}}}

\newcommand{\p}{\partial}

\newcommand{\R}{\mathbb{R}}
\newcommand{\Rn}{\R^n}

\newcommand{\supp}{\mathrm{supp}}

\newcommand{\rbar}{\overline{\mathbb R}}

\newcommand{\rp}{\mathbb R\cup\{+\infty\}}
\newcommand{\bR}{\mathbb{R}}
\def\Rn{\bR^n}

\newcommand{\IFF}{\quad\Longleftrightarrow\quad}

\newcommand{\bB}{\mathbb{B}}
\newcommand{\bN}{\mathbb{N}}
\newcommand{\bE}{\mathbb{E }}

\newcommand{\ip}[2]{\left\langle #1,\, #2\right\rangle}
\newcommand{\half}{\frac{1}{2}}

\newcommand{\set}[2]{\left\{#1\,\left\vert\; #2\right.\right\}}

\newcommand{\eg}{\emph{e.g.,}}
\newcommand{\ie}{\emph{i.e.,}}
\newcommand{\etc}{\emph{etc.}}
\newcommand{\cf}{\emph{cf.}}
\newcommand{\nb}{\emph{n.b.,}}
\newcommand{\etal}{\emph{et~al.}}

\newcommand{\lam}{\lambda}

\newcommand{\cF}{\mathcal{F}}

\newcommand{\cL}{\mathcal{L}}

\newcommand{\cV}{\mathcal{V}}

\newcommand{\AND}{\ \mbox{ and }\ }
\newcommand{\st}{\ \mbox{s.t.}\ }

\usepackage{amsopn}

\DeclareMathOperator{\E}{\mathbb{E}}
\newcommand{\reals}{\mathbb{R}}

\newcommand{\sphn}[1][n]{\mathbb{S}^{{#1}-1}}
\newcommand{\nats}{\mathbb{N}}
\newcommand{\iid}{\overset{\text{iid}}{\sim}}

\DeclareMathOperator{\gmw}{w}
\DeclareMathOperator{\rad}{rad}
\newcommand{\noise}{\ensuremath{h}}
\newcommand{\noisescale}{\ensuremath{\eta}}

\headers{LASSO reloaded}{A.\ Berk, S.\ Brugiapaglia, and T.\ Hoheisel}

\title{LASSO reloaded: a variational analysis perspective with applications to compressed sensing%
  \thanks{%
    Submitted to the editors DATE.%
    \funding{%
      The first author was partially supported by a postdoc stipend from the
      {\em Centre de Recherche Math\'ematiques (CRM)} as well as the {\em
        Institut de valorisation des donn\'ees} (IVADO) and NSERC. The second
      author acknowledges the support of NSERC through grant RGPIN-2020-06766,
      the Faculty of Arts and Science of Concordia University and the CRM. The
      third author was partially supported by the NSERC discovery grant
      RGPIN-2017-04035.}}}

\author{Aaron Berk%
  \thanks{Dept.\ Math \& Stats,
    McGill University,
    Montr\'eal, QC, Canada
    (\email{aaron.berk@mcgill.ca})}
  \and %
  Simone Brugiapaglia%
  \thanks{Dept.\ Math \& Stats,
    Concordia University,
    Montr\'eal, QC, Canada
    (\email{simone.brugiapaglia@concordia.ca})}
  \and %
  Tim Hoheisel%
  \thanks{Dept.\ Math \& Stats,
    McGill University,
    Montr\'eal, QC, Canada
    (\email{tim.hoheisel@mcgill.ca})}}

\makeatletter
\newcommand*{\addFileDependency}[1]{
  \typeout{(#1)}
  \@addtofilelist{#1}
  \IfFileExists{#1}{}{\typeout{No file #1.}}
}
\makeatother

\ifpdf
\hypersetup{
  pdftitle={LASSO reloaded: a variational analysis perspective with applications to
    compressed sensing},
  pdfauthor={A. Berk, S. Brugiapaglia, and T. Hoheisel}
}
\fi

\begin{document}

\maketitle

\begin{abstract}
  This paper provides a variational analysis of the unconstrained formulation of the LASSO problem, ubiquitous in statistical learning, signal processing, and inverse problems. In particular, we establish smoothness results for the optimal value as well as Lipschitz and smoothness properties of the optimal solution as functions of the right-hand side (or \emph{measurement vector}) and the regularization parameter. Moreover, we show how to apply the proposed variational analysis to study the sensitivity of the optimal solution to the tuning parameter in the context of compressed sensing with subgaussian measurements. Our theoretical findings are validated by numerical experiments.
\end{abstract}



\section{Introduction}\label{sec:Intro}

\noindent
One of the most important problems in the applied mathematical sciences is to
recover a signal $ x_0\in \Rn$ from noisy linear measurements
$b = A x_0 + h\in\mathbb{R}^{m}$, where $A\in \R^{m \times n}$ is a measurement
(or sensing) matrix and $h\in \R^m$ is a noise vector.  A fundamental
observation is that such a \emph{linear inverse problem} can be assumed (or cast
to) have \emph{sparse solutions}, which can be recovered with high probability
from $m \ll n$ (random) observations via computationally efficient signal
reconstruction strategies. This is well documented in the groundbreaking work by
Donoho~\cite{Don 06} and Cand\`es, Romberg, and Tao~\cite{CRT 06, CaT 05}, which
gave rise to the field of \emph{compressed sensing}. Since its introduction, the
compressed sensing paradigm led to major technological advances in a vast array
of signal processing applications, such as, most notably, compressive
imaging. For an introduction to the field, its applications, and historical
remarks, we refer the reader to \cite{AdH 21, ElK 12, FoR 13, LaW 21, Vid 19}.

In the noiseless setting (\ie{} when $h = 0$), the sparse recovery paradigm for
linear inverse problems manifests itself in the optimization framework
\begin{equation}
\label{eq:Sparse}
\min_{x\in \R^n} \|x\|_0\st A x=b,
\end{equation}
where $\|\cdot\|_0$ is counting the nonzero entries of a vector in
$\R^n$. Despite the absence of noise, problem~\cref{eq:Sparse} is provably
NP-hard in general~\cite{FoR 13, Nat 95}. Thus, many convex relaxations of this
optimization problem have been proposed, all of which, in essence, rely on the
fact that the $\ell_1$-norm is the \emph{convex envelope} of $\|\cdot\|_0$
restricted to an $\ell_\infty$-ball (see, \eg{}~\cite[\S D.4]{AdH 21}).

Here we focus on sparse recovery via $\ell_1$ minimization in the noisy setting (\ie{} when $h\neq 0$) based on the well-known \emph{(unconstrained) LASSO (Least Absolute Shrinkage and Selection Operator)} problem   
\begin{equation}\label{eq:LASSO} 
  \min_{x\in \R^n} \frac{1}{2}\|Ax-b\|^2+\lambda \|x\|_1,
\end{equation}
where $\lambda>0$ is a regularization (or tuning) parameter and where
$\|\cdot\|$ and $\|\cdot\|_1$ denote the $\ell_2$- and $\ell_1$-norm,
respectively. To the best of our knowledge, the LASSO problem was originally proposed by Santosa and Symes \cite{SaS 86} and then introduced (in a constrained
formulation) by Tibshirani in the context of statistical regression \cite{Tib 96}. Since then, the LASSO has become an
indispensable tool in statistical learning, especially when performing tasks
such as model selection (see \cite{HTW 15} and references therein). Moreover, it
plays a key role in Bayesian statistics, thanks to its ability to characterize
maximum \emph{a posteriori} estimators in linear regression when the parameters
have independent Laplace priors \cite{TrC 08}. We also note that problem
\cref{eq:LASSO} was introduced in signal processing by Chen, Donoho, and
Saunders in \cite{CDS 01} under the name of \emph{basis pursuit denoising}. Here
on, we refer to the unconstrained LASSO simply as the LASSO.

From the optimization perspective, the LASSO falls into the category of
(additive) composite problems, for which many numerical solution methods have
been devised and which have been tested on \cref{eq:LASSO}, including proximal
gradient methods (\eg{} FISTA~\cite{BeT 09}) and primal-dual methods, see Beck's
excellent textbook~\cite{Bec 17} for references, or proximal Newton-type methods
proposed by Lee \etal{}~\cite{LSS 14}, Khanh~\etal{}~\cite{KMP 21}, Kanzow and
Lechner~\cite{KaL 21} or Milzarek and Ulbrich~\cite{MiU 14}.

From the perspective of \emph{variational analysis}~\cite{BoS 00, DoR 14, Mor
  06, Mor 18, RoW 98}, given any optimization problem with parameters, the
question as to the behavior of the optimal value and the optimal solution(s) as
functions of the parameters arises naturally. The trifecta for solutions of any
optimization problem is: \emph{existence, uniqueness} and \emph{stability}. For
the LASSO problem, existence is easily established as the $\ell_1$-norm is
\emph{coercive} (and the quadratic term is not
\emph{counter-coercive}\footnote{In the sense of \cite[Definition 3.25]{RoW
    98}.}).  Sufficient conditions for uniqueness were established by Tibshirani
\cite{Tib 13} and Fuchs \cite[Theorem 1]{Fuc 04}.  A set of conditions (see
\Cref{ass:Weak} below) that characterizes uniqueness of solutions of a whole
class of $\ell_1$-optimization problems (including the LASSO) were established
by Zhang \etal{}~\cite{ZYC 15}.  An alternative, shorter proof (even though it
is not explicitly stated for the LASSO) of this characterization was given by
Gilbert \cite{Gil 17} which relies, in essence, on polyhedral convexity.
Stability results for the LASSO are somewhat scattered throughout the
literature.  Previous work has examined sensitivity of various formulations of
$\ell_1$ optimization techniques with respect to the choice of the tuning
parameter; and other work has examined their robustness to, \eg{} measurement
error~\cite{AAB 19,FTZ 20}. Regarding the selection of the tuning parameter, the
choice of the optimal parameter has been well-studied. In the case of LASSO, the
optimal choice of tuning parameter was analyzed in~\cite{Bickel 09}
and~\cite{Shen 15} in such a way as to yield a notion of stability for all
sufficiently large $\lambda$. Other work has characterized the recovery error
for LASSO in terms of the tuning parameter, but does not discuss notions of
sensitivity~\cite{Bay 11,Oymak 2013,Thramp 2015,Thramp 2018}. An asymptotic
result establishing sensitivity of the error when $\lambda$ is less than the
optimal choice has been exhibited in a closely related simplification~\cite{Berk
  2020}. Notions of sensitivity of the recovery error with respect to variation
of the tuning parameter have been discussed in previous work for other
formulations of the LASSO program~\cite{Berk 2021}.  The work by Vaiter
\etal{}~\cite{VPF 15, VDF 17} (based on partial smoothness~\cite{Lew 02})
contains the only more systematic account (and for more general regularizers)
but is confined to the stability in the right-hand side. The
variational-analytic perspective that we take, based on set-valued implicit
function theorems based on graphical and coderivatives is new and yields an
array of results derived in a uniform fashion.

\subsection*{Main contributions} 
The main contributions of this paper are the following:  

\begin{itemize}
\item We establish (in \Cref{prop:RLS}) the smoothness of optimal value functions for general regularized least-squares problems, which encompass the LASSO problem \cref{eq:LASSO} as a special case.

\item  We demonstrate  (in \Cref{ex:Counter}) that the conditions established by Zhang \etal{}~\cite[Condition 2.1]{ZYC 15}, which characterize uniqueness   of solutions for the LASSO problem \cref{eq:LASSO},  do not generally  suffice to obtain a locally  single-valued, let alone locally Lipschitz, solution function. 
\item Under an assumption  which has been previously used  to establish uniqueness  \cite{Tib 13},
we prove (in \Cref{th:LASSOSol}) local Lipschitz continuity of the solution map 
\[
  S:(b,\lambda)\in\R^n\times \R_{++}\mapsto\argmin_{x\in \R^n}\left\{\frac{1}{2}\|Ax-b\|^2+\lambda \|x\|_1\right\}
\]
with an explicit Lipschitz bound. Under a slightly stricter assumption which has also occurred in the literature as sufficient for uniqueness \cite[Theorem 1]{Fuc 04}, we prove that the solution map is continuously differentiable at the point $(\bar b,\bar \lambda)$ of questions, and establish an improved Lipschitz bound $L$ at $(\bar b,\bar \lambda)$. This bound for $S$, when considered  only as a function in $\lambda$,  reads (\Cref{cor:Lipschitz_lambda})
\[
  L\leq \frac{\sqrt{|I|}}{\sigma_{\min}(A_I)^2}
\]
where $I$ is the support of $\bar x=S(\bar b,\bar \lambda)$. 
\item As an intermediate step of our analysis we prove
  (see~\cref{prop:lasso-is-metrically-regular}) \emph{(strong) metric
    regularity} of the subdifferential operator of the objective function
  $\varphi:=\frac{1}{2}\|A(\cdot)-b\|^2+\lambda\|\cdot\|_1$. The metric
  regularity of $\varphi$ is of independent interest, as it is used in,
  \eg{}~\cite{KMP 21} to establish convergence of a numerical method for solving
  the LASSO problem. In~\cref{exa:Fuchs}, we provide further insights on the
  assumptions used to prove these results and on the sharpness of the resulting
  Lipschitz constant bound under perturbations of $\lambda$
  (see~\cref{cor:Lipschitz_lambda}).

\item Finally, motivated by compressed sensing applications, we show how to apply these results to study the sensitivity of  LASSO solutions to the tuning parameter $\lambda$ when $A$ is a subgaussian random matrix and $m \ll n$. This is first addressed in~\cref{prop:applic_variational_lambda}, under an additional assumption on the sparsity of the LASSO solution. In~\cref{prop:applic_variational_lambda_no_sparsity} we show how to remove this assumption when $b = A x_0 + h$ for some $s$-sparse vector $x_0\in\mathbb{R}^n$ and under a bounded noise model for $h$. We also validate our theoretical findings with numerical experiments in~\Cref{sec:numerical}.
\end{itemize}

\subsection*{Roadmap}

The rest paper is organized as follows: In \Cref{sec:Prelim}, we provide the
background from variational and convex analysis necessary for our
study. \Cref{sec:Value} is devoted to the convex analysis of optimal values for
regularized least-squares problems as a function of the regularization parameter
and the right-hand side (or measurement vector). In turn, in
\Cref{sec:Solution}, we study the optimal solution(s) of the LASSO problem as a
function of the regularization parameter and the right-hand side through the
lens of variational analysis.  \Cref{sec:Appli} brings the findings from the
previous section to bear on compressed sensing with subgaussian random
measurements.  We close with some final remarks in \Cref{sec:Final}.

\subsection*{Notation}
\label{sec:notation}

We write $\R_+$ for the nonnegative real numbers, $\R_{++}$ for the positive
real numbers, and $\rbar:=\R\cup\{\pm \infty\}$ for the extended real line. For
a scalar $t\in \R$, its \emph{sign} is denoted by $\sgn(t)$. For a vector
$x\in \R^n,$ the operation is to be applied component-wise, i.e.\
$\sgn(x)=(\sgn(x_i))_{i=1}^n$.  The \emph{support} of the vector
$x\in \R^n$ is
$\supp(x)=\set{i\in\{1,\dots,n\}}{x_i\neq 0}$. The set of all
linear maps from the Euclidean space $\bE$ into another $\bE'$ will be denoted
by $\cL(\bE,\bE')$.
For a matrix $A\in \R^{m\times n}$ and $I\subseteq\{1,\dots,n\}$, we denote by $A_I\in \R^{m\times |I|}$ the matrix composed of the columns of $A$ corresponding to $I$.  On the other hand, for $y\in \R^n$ we write $y_I$ for the vector in $\R^{|I|}$ whose entries correspond to the indices $i\in I$, so that the product $A_Iy_I$ can be formed.
Moreover, we denote the $i$th column of $A$ by $A_i$. 
For $F:\bE_1\to\bE_2$ differentiable at $\bar x\in \bE_1$, we write $DF(\bar x)$ for its derivative. If $G:\bE_1\times \bE_2\to \bE_3$, $\bar y\in \bE_2$  and $F:=G(\cdot,\bar y)$ is differentiable at $\bar x$, we write $D_xG(\bar x, \bar y) := DF(\bar x)$. This is extended analogously to product spaces with more than two factors.

\section{Preliminaries}\label{sec:Prelim}

\noindent
In what follows, let   $(\bE,\ip{\cdot}{\cdot})$ be a Euclidean, i.e.\ a finite-dimensional real inner product space.  For our purposes, $\bE$ will  be a product space of the form $\bE=\R^n\times \R^m\times \R^p$, whose inner product is the sum of the standard Euclidean inner products of the respective factors.  We equip  $\bE$ with the Euclidean norm derived from the inner product through $\|x\|:=\sqrt{\ip{x}{x}}$ for all $x\in\bE$. The induced operator norm of $ L\in \cL(\bE,\bE')$ is also denoted by $\|\cdot\|$ and given by
$\|L\|:=\max_{\|x\|\leq 1}\|L(x)\|=\max_{\|x\|=1}\|L(x)\|$.
Moreover, we denote the unit $\ell_p$-ball of $\mathbb{R}^n$ as $\mathbb{B}_p$, where the dimension $n$ of the ambient space will be clear from the context. We denote the minimum\footnote{We point out that for $B\neq 0$, we refer to the smallest \emph{positive} singular value as the minimum singular value.}  and maximum singular values of a matrix $B\in \R^{m\times p}$ by $\sigma_{\min}(B)$ and $\sigma_{\max}(B)$, respectively.  Recall that the operator (or spectral) norm  of $B\in \R^{m\times p}$ is its largest singular value $\sigma_{\max}(B)$ \cite{HoJ 13}. In particular, the following result holds.

\begin{lemma}\label{lem:ON} Let $B\in \R^{m\times (n-r)}$ and let $[B\; 0]\in \R^{m\times n}$. Then $\|[B\; 0]\|=\|B\|.$
\end{lemma}
\begin{proof} This follows immediately from the fact that the singular values of $B$ and $[B\; 0]$ are the square roots of the eigenvalues of $BB^T$.
\end{proof}

\subsection{Tools from  variational analysis}
\label{sec:tools-variational-analysis}
We provide in this section the necessary tools from variational analysis, and we
follow here the notational conventions of Rockafellar and Wets~\cite{RoW 98},
but the reader can find the objects defined here also in the books by
Mordukhovich~\cite{Mor 06, Mor 18} or Dontchev and Rockafellar~\cite{DoR 14}.

Let $S:\bE_1\rightrightarrows\bE_2$ be a set-valued map. The domain and graph of $S$, respectively, are the sets $\dom S:=\set{x \in \bE_1}{S(x)\neq \emptyset}$ and $\gph S:=\set{(x,y)\in \bE_1\times \bE_2}{y\in S(x)}$. The \emph{outer limit} of $S$ at $\bar x \in \bE_1$  is
$\Limsup_{x\to \bar x} S(x):=\set{y\in\bE_2}{\exists \{x_k\}\to \bar x,  \{y_k\in S(x_k)\}\to y}$.

Now let $A\subseteq  \bE$.  The \emph{tangent cone} of $A$ at $\bar x\in A$ is $ T_A(\bar x):= \Limsup_{t\downarrow 0} t^{-1}\left(A-\bar x\right)$.
The \emph{regular normal cone} of $A$ at $\bar x\in A$ is the polar of the tangent cone, \ie{}
\[
  \hat N_A(\bar x):=T_A(\bar x)^\circ=\set{v \in \bE}{\ip{v}{y}\leq 0,\;\;\forall y\in T_{A}(\bar x)}.
\]
The \emph{limiting normal cone}   of $A$ at $\bar x\in A$ is 
$N_A(\bar x):=\Limsup_{x\to \bar x} \hat N_A(x)$.
The \emph{coderivative} of $S$ at $(\bar x,\bar y)\in \gph S$ is the map $D^*S(\bar x\mid\bar y):\bE_2\rightrightarrows\bE_1$ defined via
\begin{equation}
  \label{eq:def_coderivative}
  v\in D^*S(\bar x\mid\bar y)(y) \IFF (v,-y)\in N_{\gph S}(\bar x,\bar y).
\end{equation}
The    \emph{graphical derivative} of $S$ at $(\bar x,\bar y)$ is the map $DS(\bar x\mid\bar y):\bE_1\rightrightarrows \bE_2$ given by
\begin{equation}
  \label{eq:def_graphical_derivative}
  v\in DS(\bar x \mid \bar y)( u) \IFF (u,v) \in T_{\gph S}(\bar x,\bar y),
\end{equation} 
or, equivalently (\cite[Eq.~8(14)]{RoW 98}), 
\begin{equation}\label{eq:GD}
  DS(\bar x \mid \bar y)( u)=\Limsup_{t\downarrow 0, \, u'\to u} \frac{S(\bar x+tu')-\bar y}{t}.
\end{equation}
The \emph{strict graphical derivative} of $S$ at $(\bar x,\bar y)$ is  $D_*S(\bar x\mid\bar y):\bE_1\rightrightarrows \bE_2$ given by 
\begin{eqnarray*}
  D_*S(\bar x\mid\bar y)(w)=
  \set{z \in \bE_1}{\exists \left\{\begin{array}{l}
    \vspace{0.1cm}\{t_k\}\downarrow 0,\, \{w_k\}\to w,\\ 
    \{z_k\}\to z,\\
    \{(x_k,y_k)\in\gph S\}\to (\bar x,\bar y) \end{array}\right\}: z_k\in\frac{S(x_k+t_kw_k)-y_k}{t_k}}.
\end{eqnarray*}
We adopt the convention to set $D^*S(\bar x):=D^*S(\bar x\mid\bar y)$ if $S(\bar x)$ is a singleton, and  proceed  analogously for the graphical derivatives. 

We point out that if $S$ is single-valued and continuously differentiable at $\bar x$, then $DS(\bar x)=D_*S(\bar x)$ coincides with its derivative at $\bar x$. Moreover, in this case $D^*S(\bar x)=DS(\bar x)^*$. Therefore, there is, in this case, no ambiguity in notation.

More generally, we will employ the following sum rule for the derivatives introduced above frequently in our study in \Cref{sec:Solution}. 

\begin{lemma}[{\cite[Exercise 10.43 (b)]{RoW 98}}]\label{lem:Sum} Let $S=f+F$ for $f:\bE_1\to\bE_2$ and   $F:\bE_1\rightrightarrows \bE_2$. Let $(\bar x,\bar u)\in \gph S$ and assume that $f$ is continuously differentiable at $\bar x$. Then: 
  \begin{itemize}
  \item [(a)] $DS(\bar x|\bar u)(w)=Df(\bar x)w+DF(\bar x|\bar u-f(\bar x))(w),\quad \forall w\in \bE_1$;
  \item[(b)]  $D_*S(\bar x|\bar u)(w)=Df(\bar x)w+D_*F(\bar x|\bar u-f(\bar x))(w),\quad  \forall w\in \bE_1$;
  \item[(c)]  $D^*S(\bar x|\bar u)(y)=Df(\bar x)^*y+D^*F(\bar x|\bar u-f(\bar x))(y),\quad \forall y\in \bE_2$.
  \end{itemize}
\end{lemma}
\noindent
\subsection{Tools from convex analysis} 
For well-known terms and objects in convex analysis (proper, closed, lower semicontinuous (lsc), epigraph, \etc{}) we refer to~\cite{bauschke2011convex, Roc 70}. We set
$\Gamma_0(\bE):=\set{f:\bE\to\rp}{f\;\text{closed, proper, convex}}$.
The \emph{(Fenchel) conjugate} $f^*$ of $f$  is given by $f^*(y) := \sup_{x\in\bE}\{\ip{y}{x}-f(x)\}$.
As it will occur frequently in our study, we point out that 
$\|\cdot\|_1^*=\delta_{\bB_\infty}$, i.e.\ the conjugate of the $\ell_1$-norm is the indicator function of the $\ell_\infty$-ball. Here,  given a set $S \subseteq \bE$, the \emph{indicator function} of $S$ is denoted $\delta_S(x)$ and it is equal to $0$ if $x\in S$ and $+\infty$ otherwise. The  \emph{(convex) subdifferential} of $f$ at $\bar x\in \bE$ is $\p f(\bar x):=\set{v\in \bE}{f(\bar x)+\ip{v}{x-\bar x}\leq f(x) \;\forall x\in \bE}$, which is always closed and convex, possibly empty (even for $\bar x\in \dom f$). 
An alternative description is 
$\p f(\bar x)=\set{y\in \bE}{(y,-1)\in N_{\epi f}(\bar x,f(\bar x))}$.
In particular, the \emph{relative interior} \cite[Chapter 6]{Roc 70} of the subdifferential is characterized by (\eg{} see~\cite[Theorem 6.8]{Roc 70})
\begin{equation}\label{eq:RIN}
  y\in \ri(\p f(\bar x)) \IFF (y,-1)\in \ri N_{\epi f}(\bar x, f(\bar x)).
\end{equation}
Note that the subdifferential of the indicator of a convex set $S\subseteq\bE$
is the normal cone to $S$, \ie{} $\p \delta_S=N_S$. The subdifferential operator
induces a set-valued map $\p f:\bE\rightrightarrows\bE$ which, for
$f\in \Gamma_0(\bE)$, has closed graph and nonempty domain contained in
the domain of $f$. An important example for our study is the $\ell_1$-norm
$\|\cdot\|_1:\R^n\to\R$. In this case, we have
\begin{equation}
  \label{eq:SubL1}
  \p\|\cdot\|_1(x)%
  =\prod_{i=1}^n
  \left.
    \begin{cases}
      \sgn(x_i),& x_i\neq 0
      \\
      [-1,1],& x_i=0
    \end{cases}
  \right\}
  =\set{y\in \bB_\infty}{\ip{x}{y} = \|x\|_1}, \quad \forall x\in \mathbb{R}^n.
\end{equation}
The central result that we will use to study optimal value functions of  parameterized convex optimization problems is the following. 

\begin{theorem}[Conjugate and subdifferential of optimal value function]\label{th:InfProj}
For a function $\psi\in \Gamma_0(\bE_1\times \bE_2)$, the optimal value function
\begin{equation}\label{eq:inf-proj}
  p:x\in\bE_1\mapsto\inf_{u \in \mathbb{E}_2}\, \psi(x,u)
\end{equation}
is convex and the following hold:
\begin{itemize}
\item[(a)] $p^*=\psi^*(\cdot,0)$,  which is  closed and  convex;
\item[(b)] for  $\bar x \in \mathbb{E}_1$ and $\bar u\in \argmin \psi(\bar x,\cdot)$, we have $\p p(\bar x)=\set{v \in \mathbb{E}_1}{(v,0)\in \p \psi(\bar x,\bar u)}$;
\item[(c)] $p^*\in \Gamma_0(\bE_1)$ if and only if
$\dom \psi^*(\cdot,0)\neq\emptyset$;
\item[(d)]  $p\in \Gamma_0(\bE_1)$ if $\dom \psi^*(\cdot,0)\neq\emptyset$, hence the infimum in \cref{eq:inf-proj} is attained when finite.
\end{itemize}
\end{theorem}
\begin{proof} All statements, except (b), can be found in, \eg{}~\cite[Theorem 3.101]{Hoh 19}. Part (b) follows from (a) and the fact that, given $(x,y) \in \mathbb{E}_1 \times \mathbb{E}_2$, a pair $(r,s) \in \mathbb{E}_1 \times \mathbb{E}_2$ satisfies  $\psi(x,y)+\psi^*(r,s)=\ip{(x,y)}{(r,s)}$ if and only if $(r,s)\in \p \psi(x,y)$ (see \cite[Theorem 23.5]{Roc 70}).
\end{proof}

\section{The value function for regularized least-squares problems}\label{sec:Value}

 \noindent
 For $A\in \R^{m\times n}, b\in \R^m$ and $\lambda>0$ consider the regularized least-squares problem
 \begin{equation}\label{eq:RegLS} 
   \min_{x\in \R^n} \frac{1}{2}\|Ax-b\|^2+\lambda R(x),
\end{equation}
where $R\in \Gamma_0(\R^n)$ is a regularizer. The value function for   \cref{eq:RegLS} is 
  \begin{equation}\label{eq:ValueRegLS}
    p:(b,\lambda)\in \R^n\times \R_{+} \mapsto\inf_{x\in \R^n} \left\{\frac{1}{2}\|Ax-b\|^2+\lambda R(x)\right\}, 
 \end{equation}
 where we set $0\cdot R:=\delta_{\cl(\dom R)}$\footnote{A convention which is
   backed up by the fact that $\lambda\cdot R$ \emph{epigraphically} converges
   to $\delta_{\cl(\dom R)}$ as $\lambda \downarrow 0$. See,
   \eg{}~\cite[Proposition~4(b)]{FGH 20}.}. The next result studies this value
 function in depth. Here, note that the \emph{linear image} $A\cdot f$ of a
 function $f\in \Gamma_0(\R^n)$ under $A\in \R^{m\times n}$ is the convex
 function given by
\[
  (A\cdot f)(w)= \inf_{x\in \R^n}\set{f(x)}{x=w}, \quad \forall w\in \mathbb{R}^m,
\]
which is paired in duality with $f^*\circ A^T$, see e.g.~\cite[Theorem 16.3]{Roc
  70}. Moreover, we employ the \emph{recession function}~\cite{Roc 70} (also
called \emph{horizon function}~\cite{RoW 98}) $f^\infty\in \Gamma_0(\R^n)$ of a
convex function $f\in \Gamma_0(\R^n)$ which, given \emph{any} $\bar x\in \dom f$,
is defined by
\[
  f^\infty(x):=\sup_{t>0}\frac{f(\bar x)-f(\bar x)}{t}, \quad \forall x\in \mathbb{R}^n.
\]

\begin{proposition}[Regularized least squares]\label{prop:RLS}   The following hold for the regularized least-squares problem \cref{eq:RegLS}:

\begin{itemize}
\item[(a)] (Existence of solutions) If $R^\infty(x)>0$ for all
  $x\in \ker A\setminus\{0\}$, then \cref{eq:RegLS} has a solution for all
  $(b,\lambda)\in \R^m\times \R_{++}$.
\item[(b)] (Differentiability of value function) Let
  $(\bar b,\bar \lambda)\in \R^m\times \R_{++}$ and let $\bar x$ be a
  corresponding solution of \cref{eq:RegLS}. Then the value function $p$ in
  \cref{eq:ValueRegLS} is differentiable at $(\bar b,\bar \lambda)$ with
  \[
    \nabla p(\bar b,\bar \lambda)%
    =
    \begin{pmatrix}
      \bar b-A\bar x\\  R(\bar x)
    \end{pmatrix}.
  \] 
  Moreover, if \cref{eq:RegLS} has a solution for all $(b,\lambda)$ in some
  neighborhood $U$ of $(\bar b,\bar \lambda)$, then $p$ is continuously
  differentiable on $U$.

\item[(c)] (Continuity of value function to the boundary) Assume that
  $\rge A^T \cap \dom R\neq \emptyset$. Then $p$ is continuous at $(\bar b,0)$
  for any $\bar b\in \R^m$ in the sense that
  \[
    p(b,\lambda)\to \inf_{x\in \cl(\dom R)} \frac{1}{2}\|Ax-\bar b\|^2\quad\text{as}\quad(b,\lambda)\to (\bar b,0).
  \]
\item[(d)] (Convexity of value function) $p$ is convex as a function of $b$,
  concave as a function of $\lambda$.

\item[(e)] (Constancy of residual and regularizer value) Given
  $(\bar b,\bar \lambda)\in \R^m\times \R_{++}$, the residual $\bar b-A\bar x$
  and regularizer value $R(\bar x)$ do not depend on the particular solution
  $\bar x$.
\end{itemize}
\end{proposition}
\begin{proof}
  (a) Set $\phi:=\frac{1}{2}\|A(\cdot)-b\|^2+\lambda R\in \Gamma_0(\R^n)$. Now
  observe, \cf{}~\cite[p.~89]{RoW 98}, that
  $(\frac{1}{2}\|A(\cdot)-b\|^2)^\infty=\delta_{\ker A}-\ip{A^Tb}{\cdot}.$
  Hence, using the additivity of the horizon function operation  for convex functions with overlapping domain (see \cite[Exercise 3.29]{RoW 98}), $ \phi^\infty =\delta_{\ker A}+\ip{b}{A(\cdot)}+\lambda R^\infty, $ see
  \cite[Exercise 3.29]{RoW 98}. Consequently, using the given assumptions, we
  have $\phi^\infty(x)>0$ for all $x\neq 0$, and consequently $\phi$ is
  level-coercive by~\cite[Theorem 3.26(a)]{RoW 98}, thus admits a minimizer (see~\cite[Chapter 3D]{RoW 98}).
  \smallskip

  \noindent
  (b) For $\lambda>0$, we observe that $ p(b,\lambda)=\lambda v(b,\lambda)$,
  where
  \[
    v(b,\lambda)%
    =\inf_{x\in \R^n} \left\{\frac{1}{2\lambda}\|Ax-b\|^2+  R(x)\right\}.
  \]
  The latter fits the pattern of~\cite[Theorem~2]{FGH 20} with
  $\omega:=\frac{1}{2}\|\cdot\|^2, f:=R$ and $L(x,b)=Ax-b$. Given any
  $\bar x\in \argmin_{x\in \R^n} \left\{ \frac{1}{2}\|Ax-b\|^2+\lambda
    R(x)\right\}$ it follows from this result\footnote{Alternatively, this could
    be derived also from \Cref{th:InfProj}.} that
  \begin{eqnarray*}
    \p v(b,\lambda)& = &   \set{\left(v,-\half\|y\|^2\right)}{y=\frac{1}{\lambda}(A\bar x-b), \; -A^Ty\in \p R(\bar x), v=-y}\\
&= & \left\{ \left(\frac{b-A\bar x}{\lambda}, -\frac{1}{2}\left\|\frac{A\bar x-b}{\lambda}\right\|^2\right)\right\}.
  \end{eqnarray*}
  Therefore, since $v$ is convex, $v$ is differentiable at $(b,\lambda)$ with
  $\nabla v(b,\lambda) = \left(\frac{b-A\bar x}{\lambda},
    -\frac{1}{2}\left\|\frac{A\bar x-b}{\lambda}\right\|^2\right)^T$, see
  \eg{}~\cite[Theorem 25.2]{Roc 70}. Hence, by the product rule, we find that
  $p$ is differentiable at $ (b,\lambda)$ with
  \begin{eqnarray*}
    \nabla p(b,\lambda) 
=  \begin{pmatrix} b-A\bar x\\
     -\frac{1}{2\lambda}\|A\bar x-b\|^2
   \end{pmatrix}+
\begin{pmatrix} 0\\ \frac{1}{2\lambda}\|A\bar x-b\|^2+ R(\bar x)
\end{pmatrix}=\begin{pmatrix} b-A\bar x\\  R(\bar x)\end{pmatrix}.
  \end{eqnarray*}
  The addendum about continuous differentiability follows readily
  from~\cite[Corollary 25.5.1]{Roc 70}.

\smallskip

\noindent
(c) First note that, by the posed assumptions, we have
$A\cdot R\in \Gamma_0(\R^n)$, see \cite[Theorem 16.3]{Roc 70}. Now, using the
\emph{Moreau envelope} \cite[Definition 1.22]{RoW 98} and \emph{epigraphical
  multiplication}~\cite[Exercise 1.28]{RoW 98}, we observe that
\begin{eqnarray*}
p(b,\lambda)& = & -\inf_{y\in\R^m}\left\{ \frac{1}{2}\|y\|^2-\ip{b}{y}+\lambda\star R^*(A^Ty)\right\}\\
& = & \frac{1}{2}\|b\|^2-\inf_{y\in\R^m}\left\{  \lambda\star (R^*\circ A^T)(y)+\frac{1}{2}\|y-b\|^2\right\}\\
& = &\frac{1}{2}\|b\|^2- e_1(\lambda\star (R^*\circ A^T))(b)\\
& = & e_1(\lambda (A\cdot R) )(b)\\
&= & \lambda e_{\lambda}(A\cdot R)(b)\\
& \to & \frac{1}{2}d^2_{\cl(A\cdot \dom R)}(\bar b)\quad \text{as}\quad(b,\lambda)\to (\bar b,0),
\end{eqnarray*}
where $d_\mathcal{A}(z)$ denotes the Euclidean distance of $z$ to the set $\mathcal{A}$. Here the first identity is  due to Fenchel-Rockafellar duality~\cite[Example 11.41]{RoW 98},  the fourth uses \cite[Example 11.26]{RoW 98}, and the fifth  follows from the definition of the Moreau envelope and the fact that $\dom (A\cdot R) = A\,\dom(R)$.  The limit property
uses \cite[Proposition 4(c)]{FGH 20}. Realizing that  $\frac{1}{2}d^2_{\cl(A\cdot \dom R)}(\bar b)=\inf_{x\in \cl(\dom R)} \frac{1}{2}\|Ax-\bar b\|^2$ gives the desired statement.

\smallskip

\noindent
(d) In the proof of part (a) we saw that, for $\lambda>0$, we have  $p(b,\lambda)=\lambda v(b,\lambda)$ where $v$ is (jointly) convex (\eg{} see~\Cref{th:InfProj}). The convexity of $p(\cdot,\bar \lambda)\;(\bar \lambda>0)$ follows. On the other hand, for $\bar b\in \R^m$, $\lambda, \mu>0$ and $t\in(0,1)$, we have  
\begin{eqnarray*}
  \lefteqn{p(\bar b,t\lambda+(1-t)\mu)} \\& = & \inf_{x\in \R^n} \left\{\frac{1}{2}\|Ax-\bar b\|^2+(t\lambda+(1-t)\mu) R(x)\right\}\\
 &\geq & t\cdot\inf_{x\in \R^n} \left\{\frac{1}{2}\|Ax-\bar b\|^2+\lambda R(x)\right\}+(1-t) \cdot\inf_{x\in \R^n} \left\{\frac{1}{2}\|Ax-\bar b\|^2+\mu R(x)\right\} \\
& = & tp(\bar b,\lambda)+(1-t) p(\bar b,\mu).
\end{eqnarray*}
(e) This follows immediately from (b).
\end{proof}

\begin{remark}
 \Cref{prop:RLS} remains valid (with the appropriate adjustments) when $\frac{1}{2}\|\cdot\|^2$ is replaced by a (squared) weighted Euclidean norm    $\half \ip{V{\cdot}}{\cdot}$ for some symmetric positive definite matrix $V$. \hfill$\diamond$
\end{remark}

\noindent
Using  $R=\|\cdot\|_1$ in  \cref{eq:RegLS} we can state the following immediate result for the LASSO problem.

\begin{corollary}[LASSO]\label{cor:LASSO}  The LASSO problem \cref{eq:LASSO} always has a solution, and the following hold for its value function
\begin{equation*}
  p:(b,\lambda)\in \R^n\times \R_{+} \mapsto\inf_{x\in \R^n} \left\{\frac{1}{2}\|Ax-b\|^2+\lambda \|x\|_1\right\}:
\end{equation*}
\begin{itemize}
\item[(a)]  Let $(\bar b,\bar \lambda)\in \R^m\times \R_{++}$ and let 
$\bar x$ be a corresponding solution of \cref{eq:LASSO}. Then  $p$ is continuously   differentiable at $(\bar b,\bar \lambda)$ with 
\[
\nabla p(\bar b,\bar \lambda)= \begin{pmatrix} \bar b-A\bar x\\ \|\bar x\|_1
\end{pmatrix}.
\] 
In particular,  given $(\bar b,\bar \lambda)\in \R^m\times \R_{++}$,  the residual   $\bar b-A\bar x$ and regularizer value $\|\bar x\|_1$ do not depend on the particular solution $\bar x$.  
\item[(b)] For any $\bar{b} \in \mathbb{R}^m$, $p(b,\lambda)\to \half d^2_{\rge A}(\bar b)$ as $(b,\lambda)\to (\bar b, 0)$.
\end{itemize}
\end{corollary}

\noindent
As another special case of \cref{eq:RegLS}, we consider the case $R=\frac{1}{2}\|\cdot\|^2$ which is known as Tikhonov regularization.

\begin{corollary}[Tikhonov regularization] \label{cor:Tikhonov} 
The following hold for the value function
\[ 
  p:(b,\lambda)\in \R^n\times \R_{+} \mapsto\inf_{x\in \R^n} \left\{\frac{1}{2}\|Ax-b\|^2+\frac{\lam}{2}\|x\|^2\right\}:
\]
\item[(a)] $p$ is continuously differentiable on $\R^m\times \R_{++}$ with  
\[
\nabla p(\bar b,\bar \lambda)= \begin{pmatrix} \bar b-Ax(\bar b,\bar \lambda)\\  \frac{1}{2}\|x(\bar b,\bar \lambda)\|^2
\end{pmatrix},
\] 
where $x(b,\lambda)=(\lambda A^TA+I)^{-1}(A^Tb).$
\item[(b)] For any $\bar{b} \in \mathbb{R}^m$, $p(b,\lambda)\to \half d^2_{\rge A}(\bar b)$ as $(b,\lambda)\to (\bar b, 0)$.
\end{corollary}

\noindent
The fact that the solution map  in the Tikhonov setting can be written out explicitly as $x(b,\lambda)=(\lambda A^TA+I)^{-1}(A^Tb)$ is due to the fact that the subdifferential of the regularizer $R=\frac{1}{2}\|\cdot\|^2$ is simply the identity, which is perfectly aligned with the quadratic fidelity term. There is no explicit inversion for general $R\in \Gamma_0(\R^n)$. This provides a nice segue   to  the following section, where we study the solution map for the $\ell_1$-regularizer through implicit function theory provided by variational analysis.

\section{The solution map of LASSO}\label{sec:Solution}

\noindent
This section is devoted to the study of the optimal solution function of the LASSO problem \cref{eq:LASSO}.

\subsection{Discussion of regularity conditions}

We start by recalling that, thanks to the analysis by Zhang \etal{}
in~\cite{ZYC 15}, a solution $\bar x$ of the LASSO problem~\cref{eq:LASSO}
(given $A,b,\lambda$) is unique if and only if the following set of conditions
holds:

\begin{assumption}[{\cite[Condition 2.1]{ZYC 15}}]\label{ass:Weak} For a minimizer  $\bar x$ of \cref{eq:LASSO} and  $I:=\supp(\bar x)$,
\begin{itemize}
\item[(i)] $A_I$ has full column rank $|I|$;
\item[(ii)] there exists $y\in \R^m$ such that $ \|A_{I^C}^Ty\|_\infty<1$ and $A^T_Iy=\sgn(\bar x_I)$.
\end{itemize}
\end{assumption}

\begin{remark} The  convex-analytically inclined reader may find it  illuminating to realize that \Cref{ass:Weak} is  equivalent to the following set of conditions (see, \eg{}  Gilbert's paper~\cite{Gil 17} for an explicit proof): 
\begin{itemize}
\item[(i)] $\rge A^T+\para (\p\|\cdot\|_1(\bar x))=\R^n$;
\item[(ii)] $\ri (\p\|\cdot\|_1(\bar x))\cap \rge A^T\neq \emptyset$.
\end{itemize}
Here $\para (\p\|\cdot\|_1(\bar x))$ and $\ri (\p\|\cdot\|_1(\bar x))$ are the subspace parallel to and the relative interior of the subdifferential $\p\|\cdot\|_1(\bar x)$, respectively. In the literature~\cite{VPF 15}, condition (i) is 
referred to as \emph{source condition} or \emph{range condition}. Alternative characterizations are provided in the interesting paper by Bello-Cruz \etal{}~\cite{BLN 22} which uses methods of second-order variational analysis similar to ours.
\hfill$\diamond$
\end{remark}

\noindent
Building on an example by Zhang \etal{}~\cite[p.~113]{ZYC 15}, we will show  that the solution uniqueness guaranteed by~\Cref{ass:Weak} is not stable to perturbations in the tuning parameter $\lambda$. In particular, this shows that \Cref{ass:Weak} is not a local property. 
A stronger set of conditions, which has already occurred in the literature 
(see \cite{Tib 13}) as a sufficient condition for uniqueness, is the following:
\begin{assumption}\label{ass:Intermed} For a minimizer $\bar x$ of \cref{eq:LASSO} and 
  \begin{equation}
    \label{eq:def_equicorr_set}
    J = J(\bar{x}) :=\set{i\in\{1,\dots,n\}}{|A_i^T(b-A\bar x)|=\lambda},
  \end{equation}
  we have that $A_J$ has full column rank.
\end{assumption}

\noindent
The set $J$ in~\cref{eq:def_equicorr_set} is referred to as the \emph{equicorrelation set} in the literature~\cite{Tib 13}. Using the optimality conditions for LASSO, we note that $\supp(\bar x) \subseteq J(\bar x)$.
Moreover, a simple continuity argument shows the following: if $(A_k,b_k,\lambda_k)\to (\bar A,\bar b,\bar \lambda)\in \R^{m\times n}\times \R^m\times \R_{++}$ and $x_k\to \bar x$ solves the LASSO for 
$(A_k,b_k,\lambda_k)$ and $\bar x$ is its solution for $(\bar A,\bar b,\bar \lambda)$, then the respective equicorrelation sets satisfy for all $k$ sufficiently large
\begin{equation}\label{eq:J}
J(x_k)\subseteq J(\bar x).
\end{equation}

\noindent
In particular, we find that \cref{ass:Intermed} is a local property, \ie{} stable to small pertubations of the data.
An even stronger set of conditions, which has already occurred in the literature  (see~\cite[Theorem~1]{Fuc 04}) is the following:
\begin{assumption}
  \label{ass:Strong}
  For a minimizer $\bar x$ of \cref{eq:LASSO} and $I=I(\bar x):=\supp(\bar x)$,
  we have
  \begin{itemize}
  \item[(i)] $A_I$ has full column rank $|I|$;
  \item[(ii)] $\|A_{I^C}^T(b-A_I\bar x_I)\|_\infty<\lambda$.
  \end{itemize}
\end{assumption}

\noindent 
\Cref{ass:Strong}(ii) is referred to as a \emph{nondegeneracy condition}~\cite{VDF 17} (or~\cite[Chapter~3.3.1]{VPF 15}). 

\begin{remark}
  \label{rmk:strong-asmp-ii-equivalences}
  Given a solution $\bar x$, \Cref{ass:Strong}(ii) is equivalent to either of the two following conditions:
  \begin{enumerate}
  \item $\frac{1}{\lambda} A^T(b-A\bar x)\in \ri(\p \|\cdot\|_1)(\bar x)$;
  \item $I=J$.
  \end{enumerate}
  The first condition is a direct consequence of the optimality of $\bar x$ and the second becomes apparent after noting that $I\subseteq J$.\hfill$\diamond$
\end{remark}

\noindent
The relation between the different assumptions is clarified now.

\begin{lemma}\label{lem:AssLASSO}  It holds that
  \[
    \Cref{ass:Strong} \quad \Longrightarrow \quad \Cref{ass:Intermed} \quad  \Longrightarrow  \quad \Cref{ass:Weak}.
  \]
\end{lemma}
\begin{proof}
  The fact that \Cref{ass:Strong} implies \Cref{ass:Intermed} is clear as
  \Cref{ass:Strong}(ii) implies that $J=I$
  (see~\Cref{rmk:strong-asmp-ii-equivalences}).  As for the second implication, note that \cite[Lemma~2]{Tib 13} shows that \cref{ass:Intermed}
  implies uniqueness of the solution, and observe that uniqueness is, by~\cite[Theorem~2.1]{ZYC 15},  equivalent to~\cref{ass:Weak}.
\end{proof}

\noindent
The following result shows, in particular, that \Cref{ass:Strong} is a local property as well. 

\begin{lemma}[Constancy of support]\label{lem:ConSupp}  For $(\bar A,\bar b,\bar \lambda) \in \mathbb{R}^{m \times n} \times \mathbb{R}^m \times \mathbb{R}_{++}$ let  $\bar x$  be a minimizer of  \cref{eq:LASSO}  such that \Cref{ass:Strong}(ii) holds. Assume that $(A_k,b_k,\lambda_k)\to (\bar A,\bar b,\bar \lambda)$ and that $x_k$ is  a solution of \cref{eq:LASSO} given $(A_k,b_k,\lambda_k)$ such that $\{x_k\}\to \bar x$. Then $\supp(x_k)=\supp(\bar x)$ for all $k$ sufficiently large.
\end{lemma}
\begin{proof}
  Under the conditions posed in the lemma, set
  $\bar z:= (\bar x,\|\bar x\|_1)$, $\Omega:=\epi \|\cdot\|_1$ and
  $\phi(x,t)=\frac{1}{2\bar \lambda}\|\bar A x-\bar b\|^2+t$. Observe
  that optimality of $\bar x$ implies $\bar z\in\argmin_{(x,t)\in \Omega} \phi(x,t)$. In addition, we find
  that, by~\cref{ass:Strong}(ii) and~\cref{eq:RIN}, $\bar z$ is
  \emph{nondegenerate} (\cite[Definition~1]{BuM 88}) in the sense that
  \begin{equation}
    \label{eq:ND}
    -\nabla \phi(\bar z) %
    = \left( \frac{1}{\lambda} A^{T}(b - A\bar{x}), -1 \right) %
    \in \ri (N_{\Omega}(\bar z)).
  \end{equation}
  Similarly, for $\phi_k(x,t) := \frac{1}{2\lambda_k}\|A_k x-b_k\|^2+t$,
  we have $ z_k:=(x_k,\|x_k\|_1)\in\argmin_{(x,t)\in \Omega}\phi_k(x,t) $ for
  all $k\in \bN$. Hence, by first-order optimality \cite[Theorem 12.3]{NoW 06},
  we have $ -\nabla\phi_k(z_k)\in T_{\Omega}(z_k)^\circ $ for all $k\in \bN$. In
  particular, by Moreau decomposition~\cite[Chapter~14.1]{bauschke2011convex},
  one has $P_{T_\Omega(z_k)}(-\nabla\phi_k(z_k))=0$.  Hence, for the (1-Lipschitz)
  projection $P_{T_\Omega(z_k)}$ onto the tangent cone $T_{\Omega}(z_k)$, we
  find
  \begin{equation}\label{eq:Tangent}
    \begin{array}{rcl}
      \|P_{T_{\Omega}(z_k)}(-\nabla \phi(z_k))\|%
      & =& \|P_{T_{\Omega}(z_k)}(-\nabla \phi(z_k))-P_{T_{\Omega}(z_k)}(-\nabla \phi_k(z_k))\|\\
      &\leq &\|\nabla\phi(z_k)-\nabla\phi_k(z_k)\|\\
      &\underset{k\to \infty}{\to}& 0,
    \end{array}
  \end{equation}
  as $(A_k,b_k,\lambda_k,z_k)\to (\bar A,\bar b,\bar \lambda,\bar z)$.
  Nondegeneracy of $\bar{z}$ and polyhedrality of $\Omega$, together
  with~\cref{eq:Tangent} and~\cref{eq:ND}, are the sufficient conditions for
  applying~\cite[Corollary~3.6]{BuM 88} (with $f:=\phi$), by which we infer, for
  $k$ sufficiently large, that the \emph{active constraints} at $z_k$ are equal
  those at $\bar z$, whence $\supp(x_k)=\supp(\bar x)$
  (\emph{cf.}~\Cref{SM:Sol}).
\end{proof}
\noindent
 We now present the example announced above, which expands on an example in Zhang \etal{} \cite{ZYC 15}. 
 
 \begin{example}
   \label{ex:Counter}
   Consider the LASSO problem~\cref{eq:LASSO} with 
\[
A=\begin{pmatrix} 1 & 0 & 2\\ 0 & 2 & -2 
\end{pmatrix}
\AND 
b=\begin{pmatrix} 1\\1
\end{pmatrix}.
\]
The unique solution for $\bar \lambda=1$ (see \cite[p.~113]{ZYC 15}) is $\bar x=(0,\;1/4,\;0)^T$ with $I(\bar x)=\{2\}$. 
Indeed, we observe that  
\[
\bar x_I= 1/4,\quad   A_I=\begin{pmatrix} 0 \\2\end{pmatrix}\AND A_{I^C}^T=\begin{pmatrix}1 & 0 \\ 2 & -2  \end{pmatrix}.
\]
In particular, $A_I$ has full column rank, and  setting $\bar y:=[1/2, 1/2]^T$, we find 
\[
  A_{I}^T\bar y%
  = 1%
  = \sgn(\bar x_I) %
  \AND %
  \|A_{I^C}^T\bar y\|_\infty %
  =\left\|\binom{1/2}{0}\right\|_\infty%
  =1/2%
  <1.
\]
Therefore, the solution $\bar x$ satisfies \Cref{ass:Weak}, which confirms  its uniqueness. On the other hand, we find that 
$\|A_{I^C}^T(b-A_I\bar x_I)\|_\infty=\left\|\binom{1}{1}\right\|_\infty=1=\bar \lambda$.
Hence, \Cref{ass:Strong} is violated and, as we shall see, uniqueness of the solution is not preserved under small perturbations of $\lambda$. Indeed, for $\lambda\in (0,1)$, consider the point
$
\bar x^{\lambda}:=(1-\lambda,\; \frac{2-\lambda}{4},\;  0)^T$ and note  that $\bar{x}^{\lambda} \to \bar{x}$ as $\lambda \to \bar{\lambda}$. Then 
\[
A^Tb-A^TA\bar x^{\lambda}=
\begin{pmatrix}1\\ 2\\0\end{pmatrix}- 
\begin{pmatrix}1
-\lambda\\2-\lambda\\-\lambda \end{pmatrix}=\begin{pmatrix}\lambda\\\lambda\\ \lambda \end{pmatrix}\in \lambda \p\|\cdot\|_1(\bar x^{\lambda}),
\]
and hence $\bar x^{\lambda}$ solves the LASSO problem for any $\lambda\in (0,1).$ Moreover, we have $I(\bar x^{\lambda})=\{1,2\}$ and, hence, $\sgn(\bar x^{\lambda}_{I(\bar x^{\lambda})})=(1,\;1)^T$. Now, $A_{I(\bar x^{\lambda})}^Ty=\sgn(\bar x^\lambda_{I(\bar x^{\lambda})})$ only admits  $y=(1,\;1/2)^T$ as a solution and $\|A^T_{I(\bar x^{\lambda})^C}y \|_\infty=1$. This shows that \Cref{ass:Weak} is violated and, consequently, $\bar x^{\lambda}$ is not the unique solution to \cref{eq:LASSO}. Indeed, it can be seen that, for any $\lambda\in (0,1)$, the points
\[
\begin{pmatrix} 1-\lambda-2t\\
\frac{2-\lambda+4t}{4}\\
t
\end{pmatrix},  \quad \forall t\in \left[0, \frac{1-\lambda}{2}\right),
\]
solve the LASSO problem \cref{eq:LASSO}. 

On the other hand, for $\lambda \in (1,2)$, consider $\bar x^\lambda:=(0, \frac{2-\lambda}{4},0)$. Then 
\[
  A^Tb-A^TA\bar x^{\lambda}=
  \begin{pmatrix}1\\ \lambda\\2-\lambda\end{pmatrix}\in \lambda \p\|\cdot\|_1(\bar x^\lambda),\quad \forall\lambda\in (1,2),
\]
hence $\bar x^\lambda$ solves the LASSO problem for any $\lambda\in (1,2)$. Moreover, we see that  $A_I$ has full column rank and 
\[
  \|A_{I^C}^T(b-A_I\bar x^\lambda_I)\|_\infty
  =\left\|\begin{pmatrix}1 \\ 2-\lambda\end{pmatrix} \right\|_\infty=1<\lambda.
\]
Therefore \Cref{ass:Strong} is satisfied, and $\lambda \mapsto \bar x^\lambda$ is the solution function on $(1,2)$.
\hfill $\diamond$
\end{example}

\subsection{Variational analysis of the solution function}
\noindent
Necessary ingredients for our study are the normal and tangent  cone to the graph of the subdifferential of the $\ell_1$-norm.

\begin{lemma}[Normal and tangent cone of $\gph \p \|\cdot\|_1$]\label{lem:Normal} For $(\bar x,\bar y)\in \gph \p\|\cdot\|_1$ 
we have 
\[
  N_{\gph \p \|\cdot\|_1}(\bar x,\bar y)=\prod_{i=1}^n \begin{cases}\{0\}\times \R,& \bar x_i\neq 0, \bar y_i=\sgn(\bar x_i),\\
    \R_+\times \R_-\cup \{0\}\times \R\cup \R\times \{0\},& \bar x_i=0, \bar y_i=-1,\\
    \R_-\times \R_+\cup \{0\}\times \R\cup \R\times \{0\},& \bar x_i=0, \bar y_i=1,\\
    \R\times \{0\}, &\bar x_i=0, |\bar y_i|<1,
\end{cases}
\]
and
\[
  T_{\gph \p \|\cdot\|_1}(\bar x,\bar y)\subseteq \prod_{i=1}^n \begin{cases}\R\times \{0\},& \bar x_i\neq 0, \bar y_i=\sgn(\bar x_i),\\
    \R_{-}\times \{0\}\cup \{0\}\times \R_+,& \bar x_i=0, y_i=-1,\\
    \{0\}\times \R_-\cup \R_{+}\times \{0\} & \bar x_i=0, y_i=1,\\
    \{0\}\times \R, &\bar x_i=0, |y_i|<1.
 \end{cases}
\]

\end{lemma}
\begin{proof} Use \cite[Proposition 6.41]{RoW 98} and the separability of  $\p\|\cdot\|_1(x)=\prod_{i=1}^n \p |\cdot|(x_i)$. 
\end{proof}

\noindent 
The proof of our main result relies on some general facts about implicit
set-valued functions which are in essence covered in~\cite[Theorem 9.56]{RoW
  98}. Some more terminology is needed: a set-valued map
$S:\bE_1\rightrightarrows \bE_2$ is called \emph{metrically regular} at
$(\bar x, \bar u)\in \gph S$ if there exist neighborhoods $V$ of $\bar x$ and
$W$ of $\bar u$, respectively, and $\kappa>0$ such that
\begin{align*}
  d_{S^{-1}(u)}(x)\leq \kappa \cdot d_{S(x)}(u), %
  \quad \forall x\in V,\;u\in W.
\end{align*}
When $S$ has closed graph, metric regularity can be characterized by the Mordukhovich criterion (see~\cite[Theorem~3.3]{Mor 18}(ii)).

\begin{lemma}[Mordukhovich criterion]
  \label{lem:mordukhovich-criterion}
  Let $T : \reals^{n} \rightrightarrows \reals^{m}$ have closed graph with
  $(\bar{x}, \bar{y}) \in \gph T$. $T$ is metrically regular around
  $(\bar{x}, \bar{y})$ if and only if $\ker D^{*} T(\bar{x} \mid \bar{y}) = \{0\}$.
\end{lemma}
\noindent
If $S$ is a \emph{monotone operator}\footnote{$S$ is said to be a monotone operator if $y_1\in S(x_1),y_2\in S(x_2)$ implies $\ip{y_1-y_2}{x_1-x_2}\geq 0$.},
metric regularity (at $(\bar x, \bar u)$) is equivalent to strong
metric regularity \cite{DoR 14} (see also~\cite{AAG 08}), which means
that the inverse map $S^{-1}$ is locally Lipschitz around $\bar u$.
This is exactly our interest in this study  (and will be applicable as the operator in question is the subdifferential of a convex function).
Monotonicity permits us to leverage the rich calculus for coderivatives to verify strong metric regularity and avoid the substantially more challenging task of computing strict graphical derivatives (which is the standard option for characterizing strong metric regularity in the absence of montonicity~\cite[Theorem~4D.1]{DoR 14}).

Moreover, the map $S:\bE_1\rightrightarrows\bE_2$ is called
\emph{proto-differentiable} at $(\bar x,\bar u)\in \gph S$ if, in addition to
the outer limit in~\cref{eq:GD}, for all $u \in \bE_1$ and for any
$z\in DS(\bar x\mid\bar y)(u)$ and any $\{t_k\}\downarrow 0$ there
exist $\{u_k\}\to u$ and $\{z_k\}\to z$ such that
$ z_k\in (S(\bar x+ t_k u_k)-\bar y)/t_k $ for all $k\in \bN$~\cite[Proposition~8.41]{RoW 98}. This property is
satisfied for the maps relevant to our study: $F=\p\|\cdot\|_1$
(see~\cite[Remark 1]{FGH 20}), and hence
$x\mapsto \frac{1}{\lambda}A^T(Ax-b)+F(x)$ (\cite[Lemma 4]{FGH 20}).

\begin{proposition}
  \label{prop:Implicit}
  Let $(\bar p,\bar x)\in \R^d\times \R^n$, let
  $F:\mathbb \R^n\rightrightarrows \R^n$ be monotone (locally around $\bar x$)
  and let $f:\R^d\times \R^n\to \R^n$ be continuously differentiable at
  $(\bar p, \bar x)$ such that $f(\bar p,\cdot)$ is monotone (locally at
  $\bar x$). Define $S:\R^d\rightrightarrows\R^n$ by
  \begin{align*}
    S(p)=\set{x\in \R^n}{0\in f(p,x)+F(x)}, \quad \forall p \in \mathbb{R}^d.
  \end{align*}
  Assume $(\bar p,\bar x)\in \gph S$ (\ie{}
  $0\in f(\bar p,\bar x)+F(\bar x)$) and
  $Q : \reals^{n} \rightrightarrows \reals^{n}$ given by
  $Q := f(\bar{p}, \cdot) + F$ has closed graph with $(\bar{x}, 0) \in \gph Q$
  and $\ker D^{*}Q(\bar{x} \mid 0) = \{0\}$. Then, the following hold:
  \begin{itemize}
  \item[(a)] $Q$ is \emph{strongly metrically regular} at
    $(\bar x,0) \in \gph Q$.
  \item[(b)] $S$ is locally Lipschitz at $\bar{p}$.
  \item[(c)] If $F$ is \emph{proto-differentiable} at
    $(\bar x, -f(\bar p,\bar x))$, then the graphical derivative
    $DS(\bar p|\bar x)$ is single-valued and locally Lipschitz with
    \begin{equation}
      \label{eq:graphical_derivative}
      DS(\bar p)(q)=\set{w \in \mathbb{R}^n}{0\in DG(\bar p,\bar x|0)(q,w)}, %
      \quad \forall q \in \mathbb{R}^d,
    \end{equation}
    for $G(p,x):=f(p,x)+F(x)$. In particular, $S$ is directionally
    differentiable\footnote{\ie{}
      $S'(\bar p;d):=\lim_{t\downarrow 0}\frac{S(\bar p+td)}{t}$ exists for
      all $d\in \R^d$.} at $\bar p$ with directional derivative
    \begin{align*}
      S'(\bar p;\cdot)=DS(\bar p)(\cdot).
    \end{align*}
    In addition, $S$ is locally Lipschitz at $\bar p$ with modulus
    \begin{align*}
      L\leq\limsup_{p\to \bar p} \max_{\|q\|\leq 1} \|DS(p)(q)\|.
    \end{align*}
    If $DS(\bar p)$ is linear, then $S$ is differentiable at $\bar p$ and the
    derivative equals the graphical derivative $DS(\bar p)$.
  \end{itemize}
\end{proposition}

\begin{proof}

  (a) Since $Q$ has closed graph with $(\bar{x}, 0) \in \gph Q$ and
  $\ker D^{*}Q(\bar{x} \mid 0) = \{0\}$, it holds that $Q$ is metrically
  regular at $(\bar{x}, 0)$ by~\cref{lem:mordukhovich-criterion}. Because $Q$
  is a sum of (locally) monotone maps $f$ and $F$ and it is metrically regular
  at $(\bar x, 0) \in \gph Q$, \cite[Theorem~3G.5]{DoR 14} immediately gives
  that $Q$ is \emph{strongly metrically regular} there.

\smallskip

\noindent

(b) Let $0\in D_*G(\bar p, \bar x|0)(0,w)$ or, equivalently (applying~\Cref{lem:Sum}(b) to $D_*Q$ and using that $F$ is independent of $p$),
\[
  0\in D_xf(\bar p,\bar x)w+D_*F(\bar x|-f(\bar p,\bar x))(w)=D_*Q(\bar x|0)(w).
\]
Using (a) and the characterization of strong metric regularity via the strict graphical derivative (\cf~\cite[Theorem 9.54(b)]{RoW  98}, recalling that a strong metrically regular map has, by definition, a locally Lipschitz inverse) we have  $D_*Q(\bar x|0)^{-1}(0) = \{0\}$. This implies $w = 0$.
Hence, we can apply~\cite[Theorem 9.56(b)]{RoW 98}, which establishes
that $S$ is, in fact, locally Lipschitz.  \smallskip

\noindent
(c) Realize, by~\cite[Lemma 4]{FGH 20}, that $G$ is proto-differentiable at
$((\bar p, \bar x),0)$. Hence, observing that
$ D_*Q(\bar x|0)(w)=D_*G(\bar p,\bar x|0)(0,w), $ what was proven in (b)
and~\cite[Theorem 9.56 (c)]{RoW 98} yields that $S$ is \emph{semidifferentiable}
(\cf{}~\cite[Chapter 8H]{RoW 98}) at $(\bar b,\bar \lambda)$ and
that~\cref{eq:graphical_derivative} holds. The claim about single-valuedness
and Lipschitzness of $DS(\bar p)$ also follows from~\cite[Theorem 9.56 (c)]{RoW
  98}.

The Lipschitz bound for $S$ can be deduced from~\cite[Theorem 4B.2]{DoR 14}.
The claim about differentiability follows from~\cite[Exercise 9.25 (b)]{RoW 98} noting that \emph{calmness} is implied by local Lipschitzness (\cf{}~\cite[8(13)]{RoW 98}).
\end{proof}


A key step in applying~\cref{prop:Implicit} to the LASSO is establishing the
Mordukhovich criterion for the subdifferential map of its objective.
\begin{proposition}
  \label{prop:lasso-is-metrically-regular}
  Let $\bar{x} \in \reals^{n}$ be a solution of \cref{eq:LASSO} with
  equicorrelation set $J$ as in \cref{eq:def_equicorr_set}. Define
  $T : \reals^{n} \rightrightarrows \reals^{n}$ as
  $T(x) := \frac{1}{\lambda} A^T (Ax - b) + (\partial \|\cdot\|_1)(x)$, where
  $(A, b, \lambda)$ are as in \cref{eq:LASSO}. Then, $T$ has closed
  graph. Moreover, if \cref{ass:Intermed} holds, then
  $\ker D^{*}T(\bar{x} \mid 0) = \{0\}$ (i.e. $T$ is metrically regular at $(\bar x,0)$).
\end{proposition}
\begin{proof} Let $0 \in D^* T(\bar{x} \mid 0) (w)$. It is required to show $w=0$.
  Note $T$ is the
  subdifferential of a closed, proper, convex function, thus is (maximal)
  monotone~\cite[Theorem~12.17]{RoW 98} and has closed
  graph~\cite[Theorem~24.4]{Roc 70}. By 
  \cref{lem:Sum}(c), we have
  \begin{align*}
    D^* T(\bar x \mid 0)(w) %
    = \frac{1}{\lambda} A^T A w + D^* (\partial \|\cdot \|_1)(\bar x \mid \bar u)(w),
  \end{align*}
  where $\bar u := \frac{1}{\lambda} A^T (b - A\bar x)$. Thus,
  $0 \in D^* T(\bar x \mid 0)(w)$ if and only if
  \begin{align}
    \label{eq:ATAw-coderiv-inclusion}
    - \frac{1}{\lambda} A^T A w \in D^* (\partial \|\cdot \|_1)(\bar x \mid \bar u)(w).
  \end{align}
  Now, since $\partial \|\cdot \|_1$ is
  globally maximal monotone on $\reals^n$, if
  $z \in D^*(\partial \|\cdot \|_1)(\bar x \mid \bar u)(w)$, then, by~\cite[Theorem 5.6]{Mor 18}, we have
  $\ip{z}{w} \geq 0$. Consequently, in view
  of~\cref{eq:ATAw-coderiv-inclusion} it holds that
  \begin{align*}
    0 \leq -\frac{1}{\lambda} \ip{A^T Aw}{w} = - \frac{1}{\lambda} \| Aw \|^{2}.
  \end{align*}
  As $\lambda > 0$, this implies
  $Aw = 0$. Hence, in view of~\cref{eq:ATAw-coderiv-inclusion} we have
  $0 \in D^*(\partial \|\cdot\|_1)(\bar x \mid \bar u) (w)$, which, by definition, is equivalent to $(0,-w) \in N_{\gph \partial
    \|\cdot\|_1}(\bar x, \bar u)$. Therefore, by definition of $\bar u$ and $J$
  (\cf{}~\cref{eq:def_equicorr_set}), using \cref{lem:Normal} gives
  $w_{J^C} = 0$, whence $A_J w_J = 0$. Finally, in view of
  \cref{ass:Intermed} (namely, $\ker A_J=\{0\}$) it holds that
  $w_J = 0$. Altogether, this shows $w = 0$, thereby establishing
  the Mordukhovich criterion $\ker D^{*}T(\bar{x} \mid 0) =
  \{0\}$. 
\end{proof}

We are now in a position to prove our main result. Here, for
$M \subseteq \{1,\dots,n\}$, we let $L_M:\R^{|M|} \to \R^n$ be given by
\begin{align*}
  L_M(r)_i %
  =
  \begin{cases}
    r_i& i\in M \\
    0 & \text{otherwise}
  \end{cases},
  \quad \forall i=1,\dots, n.
\end{align*}

\begin{theorem}\label{th:LASSOSol} For
  $(\bar b,\bar \lambda)\in \R^n\times \R_{++}$ let $\bar x$ be a solution of
  \cref{eq:LASSO} with $I:=\supp(\bar x)$. Then the following hold for the
  solution map
  \[
    S:(b,\lambda)\mapsto\argmin_{x\in \R^n}\left\{\frac{1}{2}\|Ax-b\|^2+\lambda \|x\|_1\right\}.
  \] 
  \begin{itemize}
  \item[(a)] If \Cref{ass:Intermed} holds at $\bar x$,   $S$ is locally Lipschitz   at $(\bar b,\bar \lambda)$ with (local) Lipschitz modulus 
    \[
      L\leq \frac{1}{\sigma_{\min}(A_{J})^{2}} \left( \sigma_{\max} \left( A_{J}
        \right) + \left\| \frac{A_{J}^{T} (A\bar x - \bar b)}{\bar \lambda}
        \right\| \right).
\] 
Moreover, $S$ is directionally differentiable at $(\bar b,\bar \lambda)$ and the directional derivative   $S'((\bar b,\bar \lambda); (\cdot, \cdot)):\R^m\times \R \to\R^n$ is locally Lipschitz and given as follows: for $(q,\alpha)\in \R^{m}\times \R$ there exists an index set $K=K(q,\alpha)$ with $I \subseteq K \subseteq J$ such that
\[
S'((\bar b,\bar \lambda);(q,\alpha))=L_K\left((A_K^TA_K)^{-1}A_K^T\left(q+\frac{\alpha}{\bar \lambda}(A\bar x-\bar b)\right)\right).
\]

\item [(b)] If \Cref{ass:Strong} holds at $\bar x$, then 
$S$ is continuously  differentiable at $(\bar b,\bar \lambda)$ with derivative
\[
DS(\bar b,\bar \lambda)(q,\alpha)=L_I\left((A_I^TA_I)^{-1}A_I^T\left(q+\frac{\alpha}{\bar \lambda}(A\bar x-\bar b)\right)\right), 
\quad \forall (q,\alpha) \in \mathbb{R}^m \times \mathbb{R}.
\]
In particular, $S$ is locally Lipschitz at $(\bar b,\bar \lambda) $ with constant 
\[
L\leq \frac{1}{\sigma_{\min}(A_{I})^{2}} \left( \sigma_{\max} \left( A_{I} \right) + \left\| \frac{A_{I}^{T} (A\bar x - \bar b)}{\bar \lambda} \right\| \right).
\]
\end{itemize}

\end{theorem}
\begin{proof} (a) We apply \Cref{prop:Implicit} with
  $f:(\R^m\times \R_{++})\times \R^n\to \R^n, \;
  f((b,\lambda),x)=\frac{1}{\lambda}A^T(Ax-b)$, where $p = (b,\lambda)$, and
  with the set-valued map $F:\R^n\rightrightarrows \R^n,\; F=\p\|\cdot\|_1$,
  observing that they satisfy the required smoothness and monotonicity
  assumptions (as the subdifferential operator of a convex function is monotone
  \cite[Chapter 12]{RoW 98}). To simplify notation, from now on we will denote
  $f((b,\lambda),x) = f(b,\lambda,x)$ and adopt a similar convention for other
  functions depending on both $(b,\lambda)$ and $x$. Then,
  $Q := f(b, \lambda, \cdot) + F$ is precisely $T$
  from~\cref{prop:lasso-is-metrically-regular}, which is thereby metrically
  regular at $(\bar{x}, 0)$. Hence, local Lipschitz continuity follows from
  \Cref{prop:Implicit}(b).

  Now, realize that $F$ is proto-differentiable by \cite[Lemma 4]{FGH 20}. We
  can therefore apply \Cref{prop:Implicit}(c) to see that
  $DS(\bar b,\bar \lambda)$ is (single-valued) locally Lipschitz with
  \begin{equation*}
    DS(\bar b,\bar \lambda)(q,\alpha)%
    = \set{w\in \R^n}{0\in DG(\bar b,\bar \lambda,\bar x|0)(q,\alpha,w)},
  \end{equation*} 
  where $G(b,\lambda,x)=f(b,\lambda, x)+F(x)$. By \Cref{lem:Sum}, we have with
  $\bar{u} := \frac{1}{\bar{\lambda}}A^{T}(\bar{b} - A\bar{x})$,
  \[
    DG(\bar b,\bar \lambda,\bar x| 0)(q,\alpha,w)%
    = -\frac{1}{\bar \lambda} A^T \left[ q+\frac{\alpha}{\bar \lambda}(A\bar
      x-\bar b)-Aw\right] + D(\p\|\cdot\|_1)(\bar x|\bar{u})(w).
  \]
  Define the partition $I^C = I^C_< \sqcup I^C_=$, where
  \begin{equation}
    \label{eq:def_I^C_partition}
    I^C_<:=\set{i\in I^C}{|A_i^T(b- A\bar x)|< \lambda}
    \AND I^C_=:=\set{i\in I^C}{|A_i^T(b-A\bar x)|= \lambda},
  \end{equation}
  and notice that $J=I\sqcup I^C_=$.  Using \Cref{lem:Normal} and the definition
  \cref{eq:def_graphical_derivative} of the graphical derivative, we find that
\begin{eqnarray*}
\lefteqn{0\in DG(\bar b,\bar \lambda,\bar x|0)(q,\alpha,w)}\\
  & \iff &  A^T\left[q+\frac{\alpha}{\bar \lambda}(A\bar x-\bar b)-Aw\right]\in \bar \lambda D(\p\|\cdot\|_1)(\bar x|\bar{u})(w)\\
& \implies & 
\left\{\begin{array}{l}\forall i\in I:  \left(\bar\lambda w_i,A_i^T\left[q+\frac{\alpha}{\bar \lambda}(A\bar x-\bar b)-Aw\right]\right)\in \R\times \{0\},\\
 \forall i\in I^C_<:\left(\bar\lambda w_i,A_i^T\left[q+\frac{\alpha}{\bar \lambda}(A\bar x-\bar b)-Aw\right]\right)\in  \{0\}\times \R,\\
 \forall i\in I^C_=:\left (\bar\lambda w_i,A_i^T\left[q+\frac{\alpha}{\bar \lambda}(A\bar x-\bar b)-Aw\right]\right)\in \R\times \{0\}\cup \{0\}\times \R
 \end{array}\right.\\
 & \implies &  
 \left\{\begin{array}{l}w_{I^C_<}=0,\\
 A_I^T\left[q+\frac{\alpha}{\bar \lambda}(A\bar x-\bar b)-Aw\right]= 0,\\
 w_iA_i^T\left[q+\frac{\alpha}{\bar \lambda}(A\bar x-\bar b)-Aw\right]= 0, \quad \forall i\in I_{=}^C.
 \end{array}\right.
\end{eqnarray*}
Setting  $K:=\set{i\in I_{=}^C}{w_i\neq 0}\cup I$, this yields $w_{K^C} = 0$ and
\[
w_K=(A_K^TA_K)^{-1}A_K^T\left[q+\frac{\alpha}{\bar \lambda}(A\bar x-\bar b)\right].
\]
Thus, $w$ and $K$ are unique for a given $(q,\alpha)$ with  $DS(\bar b,\bar \lambda)(q,\alpha)=w$. The desired local Lipschitzness of $S$ and the directional differentiability statements in (a)  follow from \Cref{prop:Implicit}(c). 

Since \Cref{ass:Intermed} is a local property (cf.\ \cref{eq:J} and the discussion prior), we can conclude that  $S$ satisfies all the proven properties above for all $(b,\lambda)$ sufficiently close to $(\bar b,\bar \lambda)$. Hence, by reiterating the above reasoning for nearby points, $S$ is directionally differentiable at $(b,\lambda)$ sufficiently close to $(\bar b,\bar \lambda)$, and $S'((b,\lambda);(\cdot,\cdot))$ is, in particular, continuous.  Thus,   from \Cref{prop:Implicit}(c) we infer that 
\[
L:=\limsup_{(b,\lambda)\to (\bar b,\bar \lambda)} \max_{\|(q,\alpha)\|\leq 1}\|S'((b,\lambda);(q,\alpha))\|
\]
is a local Lipschitz bound for $S$ at $(\bar b,\bar \lambda)$. Now let $(b_k,\lambda_k)\to (\bar b,\bar \lambda)$ such that
\[
 \max_{\|(q,\alpha)\|\leq 1}\|S'((b_k,\lambda_k);(q,\alpha))\|\to L.    
\]
As $S'((b_k,\lambda_k);(\cdot,\cdot))$ is continuous (for all $k\in\bN$), there exists  $\{(q_k,\alpha_k)\in \bB\}\to (\bar q,\bar \alpha)\in \bB$ such that 
\[
\|S'((b_k,\lambda_k);(q_k,\alpha_k))\|\to L.
\]
Let $K_k\subseteq J_k\subseteq J$ be the associated index sets. Without loss of generality (by finiteness), we can assume $K_k\equiv K\subseteq J$. Hence, for all $k\in \bN$, we have
\begin{align*}
  \left\| S' \left( (b_{k}, \lambda_{k}); (q_{k}, \alpha_{k}) \right) \right\| %
  & = \left\| L_{K} \left( (A_{K}^{T} A_{K})^{-1} A_{K}^{T} \left( q_{k} + \frac{\alpha_{k}}{\lambda_{k}} (A S (b_{k}, \lambda_{k}) - b_{k}) \right) \right) \right\| %
  \\
  & \leq \frac{1}{\sigma_{\min}(A_{K})^{2}} \left\| A_{K}^{T} q_{k} + \frac{\alpha_{k}}{\lambda_{k}} A_{K}^{T} (AS(b_{k}, \lambda_{k}) - b_{k}) \right\|,
\end{align*}
where the inequality uses that $\|L_{K}\|\leq 1$. Passing to the limit now yields
\begin{align*}
  L %
  & \leq \frac{1}{\sigma_{\min}(A_{K})^{2}} \left\| A_{K}^{T} \bar q + \frac{\bar \alpha}{\bar \lambda} A_{K}^{T} (A\bar x - \bar b) \right\| %
  \\
  & \leq \frac{1}{\sigma_{\min}(A_{K})^{2}} \left( \max_{\|q\| \leq 1} \|A_{K}^{T} q\| + \max_{|\alpha| \leq 1} \left\| \frac{\bar \alpha}{\bar \lambda} A_{K}^{T} (A\bar x - \bar b) \right\| \right) %
  \\
  & \leq \frac{1}{\sigma_{\min}(A_{K})^{2}} \left( \sigma_{\max} \left( A_{K} \right) + \left\| \frac{A_{K}^{T} (A\bar x - \bar b)}{\bar \lambda} \right\| \right) %
  \\
  & \leq \frac{1}{\sigma_{\min}(A_{J})^{2}} \left( \sigma_{\max} \left( A_{J} \right) + \left\| \frac{A_{J}^{T} (A\bar x - \bar b)}{\bar \lambda} \right\| \right). %
\end{align*}

\noindent
(b) 
Revisiting the proof of part (a) under the additional assumption that $I=J$, yields 
\[
DS(\bar b,\bar \lambda)(q,\alpha)=L_I\left((A_I^TA_I)^{-1}A_I^T\left(q+\frac{\alpha}{\bar \lambda}(A\bar x-\bar b)\right)\right).
\]
This map is clearly linear in $(q,\alpha)$, so the differentiability follows with \Cref{prop:Implicit}(c) and the expression for the derivative follows as well. In view of (a), $S$ is (Lipschitz) continuous at $(\bar b,\bar\lambda)$, hence by  \Cref{lem:ConSupp}, there exists $\varepsilon>0$ such that  
\[\supp(S(b,\lambda))=I,\quad \forall (b,\lambda)\in \cV:=B_{\varepsilon}(\bar b,\bar \lambda).
\]
Therefore, \Cref{ass:Strong} holds at $S(b,\lambda)$ for all $(b,\lambda)\in \cV$, hence by the same reasoning as for $\bar x=S(\bar b,\bar \lambda)$,  $S$ is differentiable at $(b,\lambda)$ with
\[
DS(b,\lambda)(q,\alpha)=L_I\left((A_I^TA_I)^{-1}A_I^T\left(q+\frac{\alpha}{ \lambda}(AS(b,\lambda)-b)\right)\right),\quad \forall (b,\lambda)\in \cV.
\]
Consequently, $S$ is continuously differentiable at $(\bar b,\bar \lambda)$.

The (improved) local Lipschitz constant at $(\bar b,\bar \lambda) $ follows from
(a) together with $I=J$.
\end{proof}

\begin{remark}
  It is straightforward to show that the Lipschitz constant $L$
  in~\cref{th:LASSOSol}(a) can also be bounded as
  $L \leq \sigma_{\min}(A_{J})^{-1} \left( 1 + \| A\bar{x} - \bar{b} \|/
    \bar{\lambda} \right)$.
\end{remark}

\noindent
The following, purely conceptual example illustrates the tightness of and the transition between 
\Cref{ass:Intermed} and \Cref{ass:Strong}.

\begin{example}[Soft-thresholding operator]
  \label{exmp:soft-thresh-operator}
  Consider the the case of \cref{eq:LASSO} with $A=I$. The solution map $S$
  from \cref{th:LASSOSol} is then the \emph{proximal operator} of the
  $\ell_1$-norm (also known as the soft-thresholding operator) as a function of
  the base point $b$ and the proximal parameter $\lambda$.  It is given by
  \[
    S(b,\lambda)= \left(
      \begin{cases}b_i+\lambda,& b_i<-\lambda, \\
        0,& |b_i|\leq \lambda,\\
        b_i-\lambda,& b_i>\lambda
      \end{cases}\right)_{i=1}^n,
    \quad \forall (b,\lambda)\in \R^n\times \R_{++}.
  \]
  It is locally Lipschitz as a function of $(b,\lambda)$ which is reflected in
  \Cref{th:LASSOSol}(a) since \Cref{ass:Intermed} is satisfied off-hand. Its
  points of nondifferentiability are exactly within the set
  \[
    \cF:=\set{(b,\lambda)\in \R^n\times \R_{++}}{\exists i\in \{1,\dots,n\}: |b_i|=\lambda}.
  \]
  This illustrates nicely the (sharp) transition from part (a) to part (b) in
  \Cref{th:LASSOSol}: for a given $(\bar b,\bar \lambda)\in \R^n\times \R_{++}$
  and $\bar x:=S(\bar b,\bar \lambda)$, \Cref{ass:Strong} is satisfied at
  $\bar x$ if and only if $(\bar b,\bar \lambda) \notin\cF$. It also reflects
  the fact that the locally Lipschitz solution map is continuously
  differentiable at the point of question if and only it is \emph{graphically
    regular} there \cite[Exercise 9.25 (d)]{RoW 98}.  \hfill $\diamond$
\end{example}

\begin{remark}
  \label{rmk:A-as-parameter}
  When $A$ is considered a parameter, \ie{}
  $f(A, b, \lambda, x) = \frac{1}{\lambda} A^T (Ax - b)$, one may show
  (see~\cref{rmk:4-10-supp}) that the following holds for the solution map
  $(A, b, \lambda) \mapsto \argmin_{x\in \reals^n} \{ \frac{1}{2} \|Ax - b\|^2
  + \lambda \|x\|_1\}$.
  \begin{itemize}
  \item[(a)] Under \Cref{ass:Intermed}, $S$ is   locally Lipschitz at $(\bar A, \bar b, \bar \lambda)$ with  constant 
    \begin{align*}
      L \leq \frac{1}{\sigma_{\min}(\bar A_J)^2} \left( \sigma_{\max}(\bar A_J) (1 + \|\bar x\|) + 
      \left\|\frac{\bar A_J^T(\bar A\bar x-\bar b)}{\bar \lambda}\right\| + \|\bar A\bar x - \bar b\|\right)
    \end{align*}
  \item[(b)] Under \Cref{ass:Strong}, $S$ is   locally Lipschitz at $(\bar A, \bar b, \bar \lambda)$ with  constant 
    \begin{align*}
      L \leq \frac{\sigma_{\max}(\bar A_J) (1 + \|\bar x\|) + \sqrt{I} + \|\bar A\bar x - \bar b\|}{\sigma_{\min}(\bar A_J)^2} 
    \end{align*}
  \end{itemize}
\hfill $\diamond$
\end{remark}

\noindent
When we fix the parameter $\bar b$, and look at the solution only as a function of the regularization parameter $\lambda$, we can get a significantly sharper Lipschitz  modulus.

\begin{corollary} 
  \label{cor:Lipschitz_lambda}
  In the setting of \Cref{th:LASSOSol},  the following hold  for 
  \begin{equation}
    \label{eq:solution_map_lambda}
    S:\lambda\mapsto\argmin_{x\in \R^n}\left\{\frac{1}{2}\|Ax-\bar b\|^2+\lambda \|x\|_1\right\}.
  \end{equation} 

  \begin{enumerate}
  \item[(a)] Under \Cref{ass:Intermed}, $S$ is  locally Lipschitz with  constant 
    $
    L\leq \sqrt{|J|} \sigma_{\min}(A_J)^{-2}
    $;
  \item[(b)] Under \Cref{ass:Strong},  $S$ is  locally Lipschitz with  constant 
    $
    L\leq \sqrt{|I|} \sigma_{\min}(A_I)^{-2}.
    $
  \end{enumerate}
\end{corollary} 
\begin{proof} (a) Revisiting the proof of \Cref{th:LASSOSol}(a), for $\lambda$ sufficiently close to $\bar \lambda$ and $\alpha\in \R$ there is $K=K(\lambda,\alpha)\subseteq J$ with $S'(\lambda;\alpha)=L_K\left((A_K^TA_K)^{-1}A_K^T\left(\frac{\alpha}{\bar \lambda}(AS(\lambda)-\bar b)\right)\right)$. Hence, we find that
  \[
    \|S'(\lambda;\alpha)\|\leq  \frac{|\alpha|}{\sigma_{\min}(A_K)^2}\left\| A_K^T\left(\frac{A\bar x-\bar b}{\bar \lambda}\right)\right\|\leq \frac{|\alpha|\sqrt{|K|}}{\sigma_{\min}^2(A_K)}\leq |\alpha|\frac{\sqrt{|J|}}{\sigma_{\min}(A_J)^2}.
  \]
  Therefore 
  \[
    L=\limsup_{\lambda\to\bar \lambda} \max_{|\alpha|=1}\|S'(\lambda;\alpha)\|\leq \frac{\sqrt{|J|}}{\sigma_{\min}(A_J)^2},
  \]
  as desired. Item (b) then follows from (a) with $I = J$. 
\end{proof}

\subsection{\texorpdfstring{On~{\Cref{ass:Strong}} and the sharpness of~{\cref{cor:Lipschitz_lambda}}}{On Assumption 4.5 and the sharpness of Corollary 4.18}}
\label{sec:discussion_Fuchs}

We now discuss a case where~\Cref{ass:Strong} is satisfied and where the
single-valued solution map to the LASSO problem admits an explicit formula. This
example is due to Fuchs~\cite{Fuc 04} and is based on the notion of
\emph{coherence} (see, \eg{}~\cite[Chapter 5]{FoR 13} and references therein). We
assume for simplicity that the columns $A_1,\ldots,A_n$ of $A$ have unit
$\ell_2$-norm and recall that the coherence of $A$ is defined as
$$
\mu(A) := \max_{i\neq j}|\langle A_i, A_j\rangle|.
$$
Beyond providing more insight on~\Cref{ass:Strong}, the following example sheds
light on the sharpness of the Lipschitz bound proved
in~\cref{cor:Lipschitz_lambda}.
\begin{example}[{\cref{ass:Strong}} and sharpness of~{\cref{cor:Lipschitz_lambda}}] 
\label{exa:Fuchs} 
Assume that the columns of $A$ have unit $\ell_2$-norm and suppose that $\bar{b} \in \mathbb{R}^m$ satisfies the following assumption: there exists $x_{0}\in\mathbb{R}^n$ such that
\begin{equation}
    \label{eq:condition_on_b_Fuchs}
    A x_{0} = \bar{b}, \; \|x_{0}\|_0 < \frac12 \left(1 + \frac{1}{\mu(A)}\right)  \AND A_{I_0} \text{ has full rank with } I_0 = \supp(x_0).
\end{equation}
In other words, $Ax_{0}$ is an irreducible sparse representation of $\bar{b}$ with respect to the columns of $A$ and satisfying an explicit upper bound on the sparsity level in terms of the coherence $\mu(A)$. Under assumption \cref{eq:condition_on_b_Fuchs} and using ideas from \cite{FoR 13} it is possible to derive an explicit formula for the solution map of the {LASSO}. This allows us to check \Cref{ass:Strong} and show the sharpness of the Lipschitz bound of~\cref{cor:Lipschitz_lambda}.

Let $s_0 := (x_0)_{I_0}$ and define, for any $\lambda \geq 0$, the vector $\bar{x} \in \mathbb{R}^n$ as
\begin{equation}
  \label{eq:explicit_Fuchs}
  \bar{x}_{I_0} := s_0 - \lambda (A_{I_0}^T A_{I_0})^{-1}\sgn(s_0) \AND \bar{x}_{I_0^C}:=0.
\end{equation}
Note that, for small enough values of $\lambda>0$, $\sgn(s_0) = \sgn(\bar{x}_{I_0})$ since all entries of $s_0$ are nonzero. Hence, the quantity 
$$
\lambda_{\max} := \sup\set{\lambda^+ > 0}{\sgn(\bar{x}_{I_0}) = \sgn(s_0), \;\forall \lambda \in [0,\lambda^+)} \in (0,+\infty]
$$
is well defined and is such that $I_0=\supp(\bar{x})=:I$ for all $\lambda \in [0,\lambda_{\max})$. Moreover, recalling that $\bar{b} = A_{I_0} s_0$, for all $\lambda \in [0,\lambda_{\max})$ we have both
\begin{align}
  \frac1\lambda  A_{I_0}^T(\bar{b}-A\bar{x})
  &= \frac1\lambda  A_{I_0}^T(\bar{b}-A_{I_0}(s_0 - \lambda (A_{I_0}^TA_{I_0})^{-1}\sgn(s_0)))
  = \sgn(s_0) = \sgn(\bar{x}_{I_0}),\label{eq:optimality_Fuchs_1}
  \\
  \left|\frac1\lambda  A_{I_0^C}^T(\bar{b}-A\bar{x})\right| &= \left|A_{I_0^C}^T A_{I_0}^+ \sgn(s_0)\right|<1,   \label{eq:optimality_Fuchs_2}
\end{align}
where the last inequality is a direct consequence of \cref{eq:condition_on_b_Fuchs} and \cite[Theorem~3]{Fuc 04}. Observe that \cref{eq:optimality_Fuchs_1} and \cref{eq:optimality_Fuchs_2} imply that $\bar{x}$ is a solution of the LASSO for any $\lambda \in [0,\lambda_{\max})$. Moreover, \cref{eq:optimality_Fuchs_2}, combined with the fact that $A_{I_0}$ has full rank, yields the validity of  \Cref{ass:Strong}.

We conclude this example by showing the sharpness of the Lipschitz bound
of~\cref{cor:Lipschitz_lambda}. First, observe that the support set $I_0$
depends on $\bar{b}$ (since it determines which columns of $A$ are used to form
the sparse representation of $\bar{b}$). Yet, $I_{0}$ is independent of the
tuning parameter $\lambda$. In fact, the existence of $x_{0}$ (which, in turn,
defines $I_{0}$) in~\cref{eq:condition_on_b_Fuchs} is not related to
$\lambda$. We note, however, that $\lambda_{\max}$ does depend on $P$ and,
hence, on $\bar{b}$. This allows for a direct differentiation of the LASSO
solution map $\lambda \mapsto S(\lambda)$ with respect to $\lambda$ via the
explicit formula~\cref{eq:explicit_Fuchs} and yields, for any
$\bar{\lambda} \in (0,\lambda_{\max})$,
\begin{equation}
\label{eq:Lipschitz_lambda_Fuchs}
\left\|D_{\lambda}S(\bar{\lambda})\right\| 
= \|(A_{I_{0}}^TA_{I_{0}})^{-1}\sgn(s_0)\| 
\leq \|(A_{I_{0}}^TA_{I_{0}})^{-1}\| \|\sgn(s_0)\|
= \frac{\sqrt{|I_{0}|}}{\sigma_{\min}(A_{I_{0}})^2}.
\end{equation}
This estimate is consistent with the Lipschitz bound
of~\cref{cor:Lipschitz_lambda} since $I = I_0$. \hfill $\diamond$
\end{example}

\section{Applications and numerical experiments}
\label{sec:Appli}

In this section, we illustrate how to apply the theory presented
in~\Cref{sec:Solution} to study the sensitivity of the LASSO solution to the
tuning parameter $\lambda$ when $A$ is a random subgaussian matrix and when
$m \ll n$ (in a sense made precise below). As already mentioned in the
introduction, this case study is motivated by compressed sensing applications
\cite{AdH 21, ElK 12, FoR 13, LaW 21, Vid 19}. In this section, we
primarily restrict our focus to the Euclidean space $\reals^{n}$; and in a mild abuse
of notation, we henceforth use $\E$ to denote the expectation operator for a
random object \cite{Ver_17}. We refer to the exponential function alternately
as $e^x$ or $\exp(x)$. If $X$ is a random object
drawn from a distribution $\mathcal{D}$, we write $X \sim \mathcal{D}$. If
$\{X_i\}_{i = 1}^m$ are all independent and identically distributed
according to $\mathcal{D}$ we write $X_i \iid \mathcal{D}$, $i = 1,\ldots,m$. Throughout this section, we will let $C > 0$ denote an absolute constant, whose value may change from one appearance to the next, a presentational choice common in high-dimensional probability and compressed sensing.
Moreover,  for integers $1\leq s\leq n$, we denote $\Sigma_{s}^{n} := \{x \in \reals^{n} : \|x\|_{0} \leq s\}$, where $\|x\|_0:=|\{j\in \{1,\ldots,n\}:x_j\neq 0\}|$. Finally, we denote by $\bar x(\lambda)$ an element of the solution set
$S (\lambda)$ as it is defined in \cref{eq:solution_map_lambda} of
\Cref{cor:Lipschitz_lambda}. 

\subsection{Application to LASSO parameter sensitivity}
\label{sec:appli-lasso1}

We start by recalling some standard notions from high-dimensional
probability. For a general introduction to the topic we refer to,
\eg~\cite{Ver_17, Wai 19}.
\begin{definition}[Subgaussian random variable and vector]
  We call a real-valued random variable $X$ $\beta$-subgaussian, for some
  $\beta > 0$, if $\|X\|_{\psi_{2}} \leq \beta$ where
  \begin{align*}
    \|X\|_{\psi_{2}} := \inf \left\{ t > 0 : \E \left[\exp (X^{2}/t^{2})\right] \leq 2 \right\}. 
  \end{align*}
  A real-valued random vector $Y \in \reals^{n}$ is $\beta$-subgaussian if
  \begin{align*}
    \|Y\|_{\psi_{2}} := \sup_{\|v\| = 1} \|\ip{Y}{v}\|_{\psi_{2}} \leq \beta. 
  \end{align*}
\end{definition}

\begin{definition}[Subgaussian matrix]
  We call $A \in \reals^{m \times n}$ a $\beta$-subgaussian matrix if
  its rows $a_{1}^T, \ldots, a_m^T \in \reals^{1 \times n}$ 
  are independent subgaussian random vectors satisfying
  $\max_{i} \|a_{i}\|_{\psi_{2}} \leq \beta$ and
  $\E a_i a_i^T = I_{n}$ (where $I_{n}$ is the $n\times n$ identity
  matrix). We call $\tilde A \in \reals^{m \times n}$ a normalized
  $\beta$-subgaussian matrix if $\tilde A := m^{-1/2}A$ where
  $A \in \reals^{m\times n}$ is a $\beta$-subgaussian matrix.
\end{definition}

\noindent
The subgaussian random matrix model is popular in compressed sensing and it
encompasses random matrices with independent, identically distributed Gaussian
or Bernoulli entries. For more detail, see~\cite[Chapter 9]{FoR 13}
or~\cite[Chapter 4]{Ver_17}. It is well known that the singular values of
submatrices of subgaussian random matrices are well-behaved. This statement is
formalized as~\cref{lem:submat-least-singvals} in the supplement, and may be established
from~\cite[Corollary~1.2]{Jeo_20} and~\cite[$\S$10.3]{Ver_17}. It is used to
obtain uniform control over the minimum singular values over all $s$-element
sub-matrices $A_{I}$ where $I \subseteq \{1, 2, \ldots, n\}$ with $|I| = s$.

Such control allows us to obtain an upper bound for the Lipschitz constant $L$
in~\cref{cor:Lipschitz_lambda} that is independent of the support set $I$ and
holds with high probability when $A$ is a normalized $\beta$-subgaussian
matrix, provided that the LASSO solution is sparse enough and
that~\Cref{ass:Strong} holds.

\begin{proposition}[Sparse LASSO parameter sensitivity for subgaussian matrices]
  \label{prop:applic_variational_lambda}
  Fix integers $1 \leq s < m \leq n$ and suppose that
  $A \in \reals^{m \times n}$ is a normalized $\beta$-subgaussian
  matrix. Suppose that $\delta, \varepsilon \in (0, 1)$ and
  $m \geq C \delta^{-2} \beta^{2}\log \beta \left[ s \log (e n / s) + \log (3 /
    \varepsilon) \right]$, where $C > 0$ is the absolute constant
  from~\cref{lem:submat-least-singvals}. The following holds with probability at
  least $1 - \varepsilon$ on the realization of $A$. For
  $\bar{x}(\lambda) \in S(\lambda)$ where $S$ is the solution map in
  \cref{eq:solution_map_lambda}, if $\bar{b} \in \reals^{m}$ and
  $\bar{\lambda} > 0$ are such that $\bar{x}(\bar{\lambda})$ satisfies
  \cref{ass:Strong} and $\| \bar{x}(\bar{\lambda}) \|_{0} \leq s$, then there
  exists $r > 0$ such that for all $\lambda$ with
  $|\lambda - \bar{\lambda}| < r$,
  \begin{align*}
    \| \bar{x}(\lambda) - \bar{x}(\bar{\lambda}) \| \leq L | \lambda - \bar{\lambda}|, %
    \qquad L < \frac{\sqrt s}{(1 - \delta)^{2}}. 
  \end{align*}
\end{proposition}
\begin{proof}[{Proof of~\cref{prop:applic_variational_lambda}}]
  Let $I := \supp(\bar{x}(\bar{\lambda}))$ and note that by assumption we have $|I| \leq
  s$. By~\cref{lem:submat-least-singvals}, if $m$ satisfies the stated lower bound,
  then with probability at least $1 - \varepsilon$ on the realization of $A$ it
  holds that
  \begin{align*}
    \sigma_{\min}(A_{I}) \geq \min_{|K| \leq s} \sigma_{\min}(A_{K}) \geq 1 - \delta. 
  \end{align*}
  Here, the first inequality is trivial (since $I$ belongs to the set of indices
  with at most $s$ entries) and the second inequality follows by the referenced
  result. We restrict to this high-probability event for the remainder of the
  proof. Since $\bar{x}(\bar{\lambda})$ satisfies \cref{ass:Strong}, by
  \cref{cor:Lipschitz_lambda} the solution mapping admits a locally Lipschitz
  localization about $\bar{\lambda}$, meaning that there exist $r, L > 0$ such
  that for all $|\lambda - \bar{\lambda}| < r$ one has, as desired,
  \begin{align*}
    \| \bar{x}(\lambda) - \bar{x}(\bar{\lambda}) \| \leq L | \lambda - \bar{\lambda}|, %
    \qquad %
    L < \frac{\sqrt{|I|}}{\sigma_{\min}(A_{I})^{2}} \leq \frac{\sqrt s}{(1 - \delta)^{2}}.
  \end{align*}
\end{proof}

\Cref{prop:applic_variational_lambda} provides an upper bound to the Lipschitz
constant of the solution map $\lambda \mapsto S(\lambda)$ of the LASSO at
$\bar{\lambda}$ under~\Cref{ass:Strong} and provided that
$\bar{x}(\bar{\lambda})$ is sparse enough. We now show how to remove the
sparsity condition when the measurements are of the form
$\bar{b} = A x_{0} + \noise$ for some (approximately) sparse vector $x_{0}$ and
bounded noise $\noise$ and under a slightly stronger condition on the number of
measurements $m$. To make this possible, a key ingredient is the recent analysis
in~\cite{foucart2023sparsity} that provides explicit bounds to the sparsity of
LASSO minimizers. We also rely again on~\cref{lem:submat-least-singvals}, which
controls the \emph{restricted isometry constants} of $A$ (see,
\eg{}~\cite[Chapter 6]{FoR 13} and references therein). We begin by stating a
specialization of~\cite[Theorem~5]{foucart2023sparsity}.
\begin{lemma}
  \label{lem:foucart-2023-thm5}
  Let $0 < \alpha_{-} \leq 1 \leq \alpha_{+} < \infty$, define
  $\gamma := \alpha_{+} / \alpha_{-}$ and
  $t := \lfloor 36\gamma^{2}s\rfloor +1$. Suppose an unknown signal
  $x_{0} \in \Sigma_{s}^{n}$ is measured as $\bar b = Ax_{0} + h \in \reals^{m}$ for
  unknown noise $h \in \reals^{m}$ satisfying
  $\| h \| \leq \frac{1}{3} \| \bar b \|$ and a measurement matrix
  $A \in \reals^{m \times n}$ satisfying RIP of order $t$ with parameters $(\alpha_{-}, \alpha_{+})$
  (\emph{cf.}~\cref{def:vrip}):
  \begin{align*}
    \alpha_{-} \| x \| \leq \| Ax \|_{2} \leq \alpha_{+} \| x \|, \qquad \forall x \in \Sigma^n_t. 
  \end{align*}
  Then, for any $\lambda >  2 \alpha_{+} \| h \|$, any LASSO
  solution $\bar{x}(\lambda)$ has sparsity at most proportional to $s$, namely
  \begin{align*}
    \| \bar{x}(\lambda) \|_{0} \leq \lfloor 36 \gamma^{2} s\rfloor. 
  \end{align*}
\end{lemma}

We are now in position to state the main result of this section. 
\begin{proposition}
  \label{prop:applic_variational_lambda_no_sparsity}
  Fix $\delta, \varepsilon \in (0, 1)$, let $1 \leq s < m \leq n$ be integers
  and suppose $x_{0} \in \Sigma_{s}^{n}$. Suppose $\bar{b} := Ax_{0} + h$
  where $A \in \reals^{m \times n}$ is a normalized $\beta$-subgaussian matrix
  and $h \in \reals^{m}$ satisfies $\| h \| \leq \frac{1}{3} \| \bar b \|$. Suppose that
  \begin{align*}
    m \geq 37 C \left( \frac{1 + \delta}{\delta(1-\delta)} \right)^{2} \beta^{2} \log \beta \left[ s \log (en/s) + \log(3 / \varepsilon) \right],
  \end{align*}
  where $C > 0$ is the absolute constant from \cref{lem:submat-least-singvals}. Then, with probability at least $1- \varepsilon$ on the realization of $A$ the following holds. For
  $\bar{x}(\lambda) \in S(\lambda)$ where $S$ is the solution map in
  \cref{eq:solution_map_lambda}, if $\bar{b} \in \reals^{m}$ and
  $\bar{\lambda} > 0$ are such that $\bar{x}(\bar{\lambda})$ satisfies
  \cref{ass:Strong} and if $\bar{\lambda} > 2(1 + \delta)\| h \|$, then there exists $r > 0$ such that for all $\left| \lambda - \bar{\lambda} \right| \leq r$,
  \begin{align*}
    \| \bar{x}(\lambda) - \bar{x}(\bar{\lambda}) \| \leq L \left| \lambda - \bar{\lambda} \right|, \qquad L \leq \frac{6(1 + \delta) \sqrt s}{(1 - \delta)^{3}}. 
  \end{align*}
\end{proposition}
\begin{proof}
  Define
  $t := \lfloor 36 \gamma^{2} s \rfloor + 1$ where
  $\gamma := (1 + \delta)/(1 - \delta)$. The assumed lower bound on $m$, combined with the observation that $37 z \geq \lfloor 36 z \rfloor + 1$ for all $z \geq 1$, gives
  \begin{align*}
    m %
    & \geq 37C \left( \frac{1 + \delta}{\delta(1 - \delta)} \right)^{2} \beta^{2} \log \beta \left[ s \log (en/s) + \log (3/\varepsilon) \right]
    \\
    & \geq C\delta^{-2} \beta^{2} \log \beta \left[ t \log (en/t) + \log (3/\varepsilon) \right],
  \end{align*}
  which is sufficient for $A$ to satisfy RIP of order $t$ with parameters
  $(1-\delta, 1+\delta)$ with probability at least $1 - \varepsilon$ (by
  \cref{lem:submat-least-singvals}).  Restrict to this favorable high-probability
  event and let $I := \supp(\bar{x}(\bar{\lambda}))$. By~\cref{lem:foucart-2023-thm5},  it holds that
  $|I| = \| \bar{x}(\bar{\lambda}) \|_{0} \leq \lfloor 36 \gamma^{2}
  s\rfloor$. Consequently, since $\bar{x}(\bar{\lambda})$ satisfies
  \cref{ass:Strong}, by~\cref{cor:Lipschitz_lambda} the solution map  admits
  a locally Lipschitz localization about $\bar{\lambda}$ meaning that there
  exist $r, L > 0$ such that for all $|\lambda - \bar{\lambda}| < r$ one has
  \begin{align*}
    \| \bar{x}(\lambda) - \bar{x}(\bar{\lambda}) \| \leq L | \lambda - \bar{\lambda}|, \qquad L < \frac{\sqrt{|I|}}{\sigma_{\min}(A_{I})^{2}} \leq \frac{6(1 + \delta) \sqrt s}{(1 - \delta)^{3}}. 
  \end{align*}
\end{proof}

\begin{remark}[From sparsity to compressibility]
  Observe that the above result extends to $s$-compressible signals, \ie{} signals $x_{0} \in \reals^{n}$ for which, informally, the best $s$-term approximation error $\sigma_{s}(x_{0})_{1} := \inf \{ \|x_{0} - x\|_{1} : x \in \mathbb{R}^n,  \, \|x\|_0 \leq s\}$ is small. This can be seen by letting
$\bar b := Ax_{s} + \noise'$,  where $x_{s}$ is a best $s$-term approximation to $x_{0}$ with respect to the $\ell_{1}$-norm, i.e.\ an $s$-sparse vector such that $\|x_{0}-x_s\|_1 = \sigma_s(x_{0})_1$,  and $\noise' := A(x_{0} - x_{s}) + \noise$. We note, however, that this would require the knowledge of an upper bound to $\|A(x_{0}-x_s)\|$ in order to satisfy the assumption 
$\|\noise'\|\leq \noisescale$.\hfill$\diamond$
\end{remark}

\subsection{Numerical experiments}
\label{sec:numerical}

We conclude this section by illustrating some numerical experiments. Suppose
that $x_{0} \in \reals^{n}$ is $s$-sparse and let
$\bar b := Ax_{0} + \gamma w$ where $A \in \reals^{m \times n}$ has
$A_{ij} \iid \mathcal{N}(0, 1/m)$, $\gamma = 0.1$ and
$w_{i} \iid \mathcal{N}(0, 1)$. In particular, note that $A$ is a normalized $\beta$-subgaussian matrix for an absolute constant $\beta > 0$~\cite[Example~2.5.8 \& Lemma~3.4.2]{Ver_17}. For $\lambda > 0$ recall that
$\bar{x}(\lambda) \in S(\lambda)$ where $S$ is defined
in~\cref{eq:solution_map_lambda}, and define the best parameter choice with respect to the ground-truth:
\begin{align}
\label{eq:lambda-star}
\lambda^{*} := \inf \argmin_{\lambda > 0} \|\bar{x}(\lambda) - x_{0}\|.
\end{align}
Above, the particular choice of $\bar x(\lambda)$ is to be understood in the sense that a numerical algorithm returns a unique vector for given data, even when the solution map is set-valued.
Finally, let $I := \supp(\bar{x}(\lambda^{*}))$ and $s_{\lambda^{*}} := |I|$. In
\autoref{fig:numerics} we plot $\|\bar{x}(\lambda) - \bar{x}(\lambda^{*})\|$ as
a function of $\lambda$ (solid curve) and superpose the Lipschitz upper bound
evaluated at $\bar \lambda = \lambda^{*}$, namely
$\sqrt{s_{\lambda^{*}}} \cdot\sigma_{\min}^{-2}(A_{I})\cdot |\lambda -
\lambda^{*}|$ (dash-dot curve). Included on each plot, in view of \cref{cor:Lipschitz_lambda}(b), is the ratio of the two
quantities,
\begin{align*}
  \frac{\sqrt{s_{\lambda^{*}}} \cdot\sigma_{\min}^{-2}(A_{I})\cdot |\lambda -
\lambda^{*}|}{\|\bar{x}(\lambda) - \bar{x}(\lambda^{*})\|},
\end{align*}
providing an alternative visualization of the extent of the bound's tightness in
each setting (dotted curve). This latter curve is plotted with respect to the
axis appearing on the right-hand side of each plot.

\begin{figure}[ht]
  \centering

  $s = 3$
  
  \null\hfill\includegraphics[width=.23\textwidth]{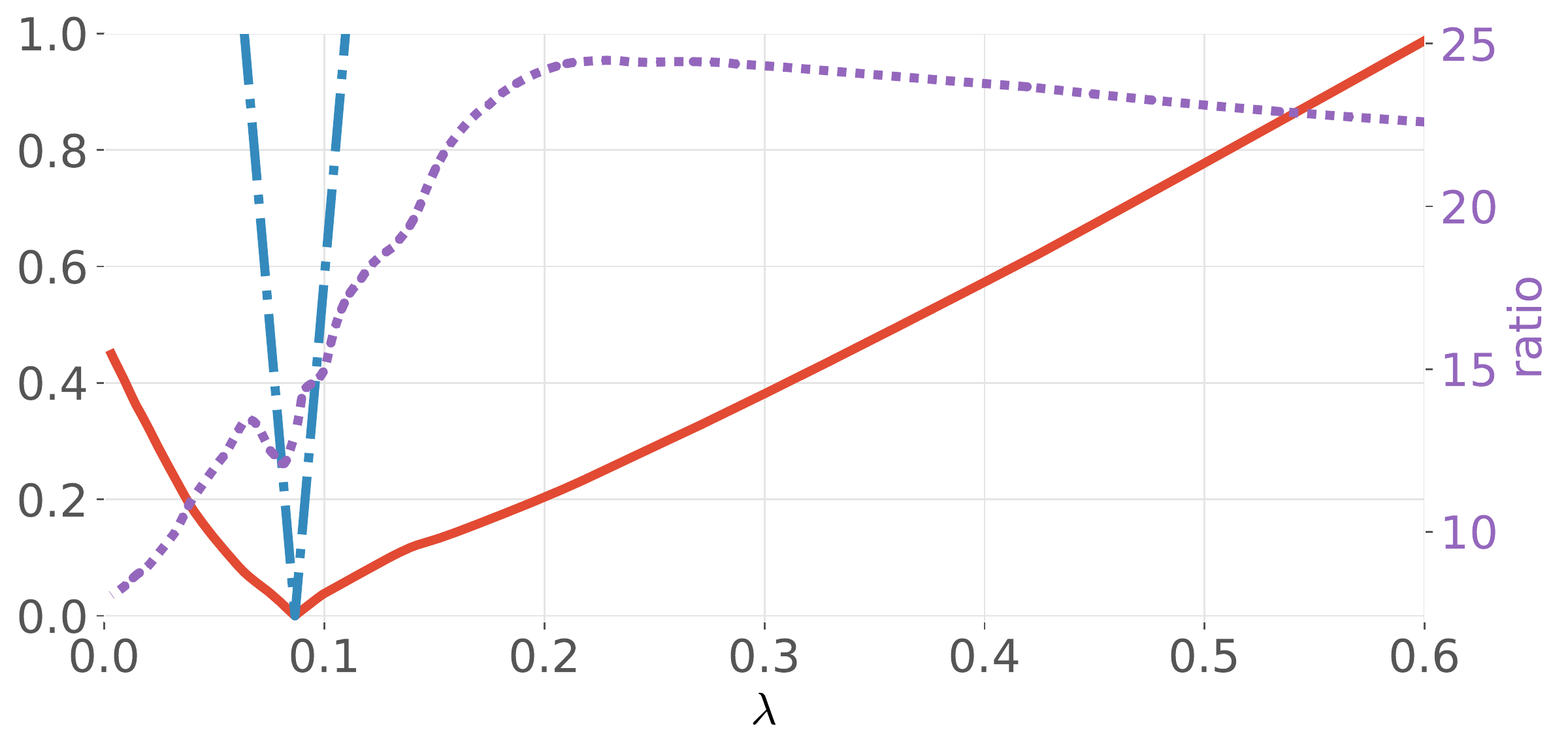}
  \hfill\includegraphics[width=.23\textwidth]{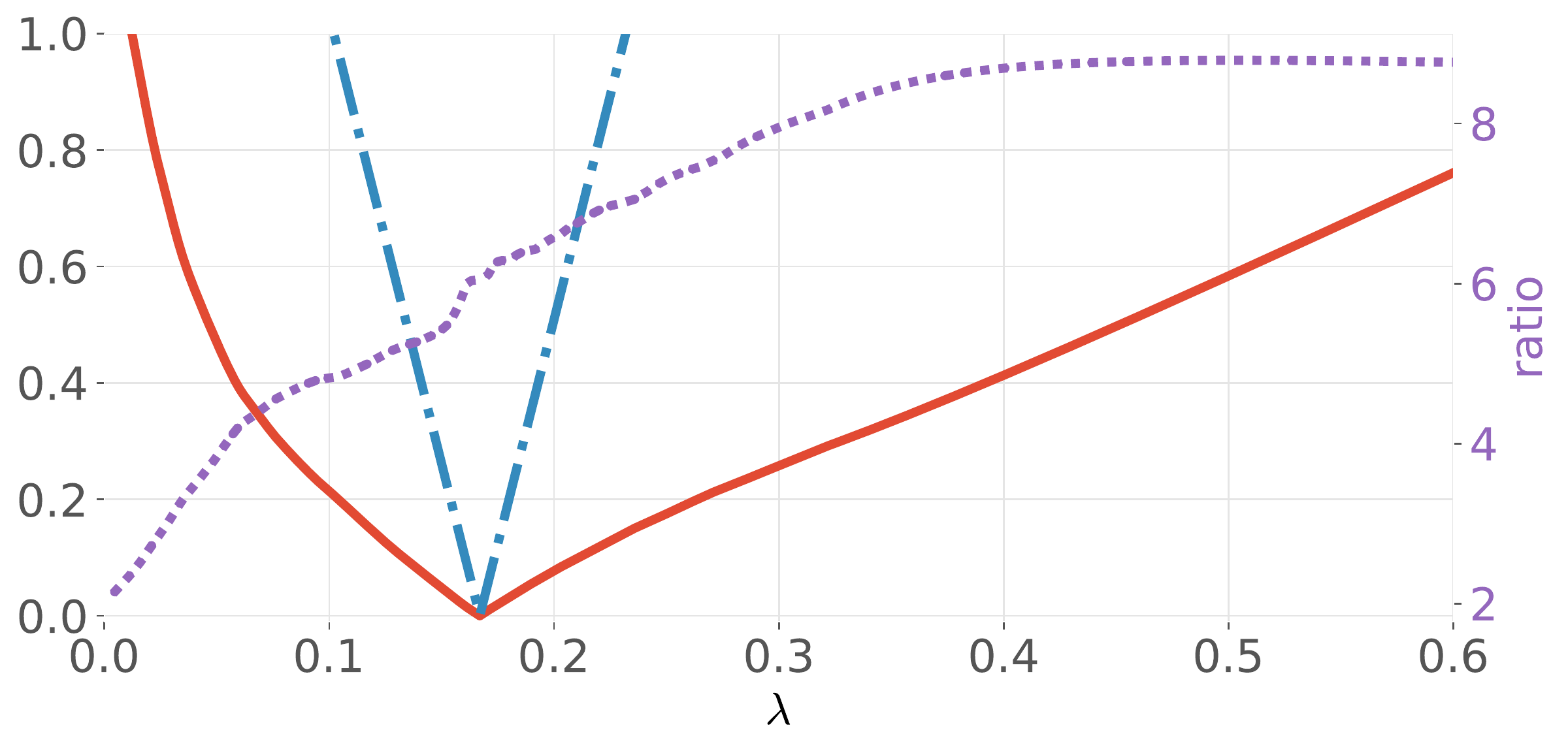}
  \hfill\includegraphics[width=.23\textwidth]{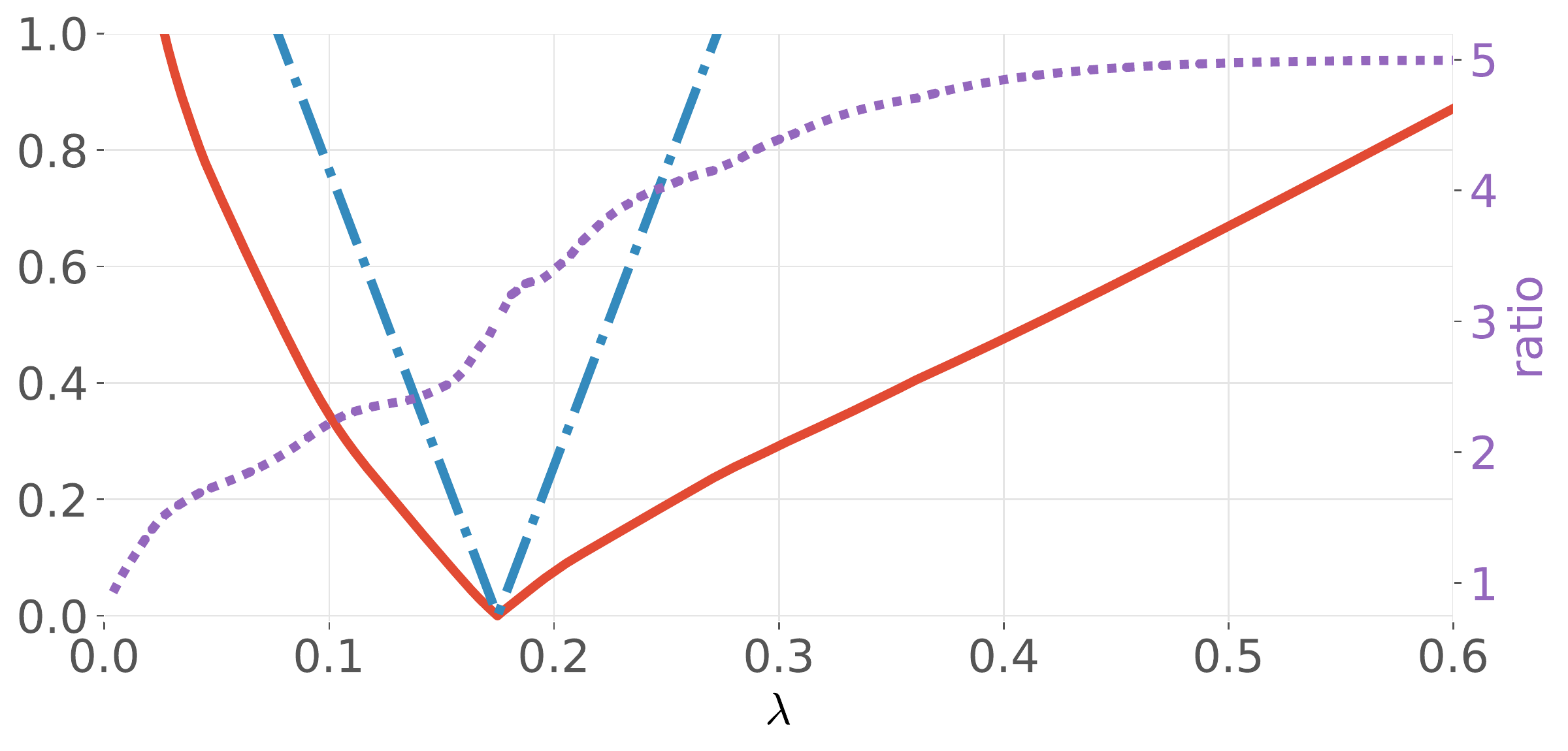}
  \hfill\includegraphics[width=.23\textwidth]{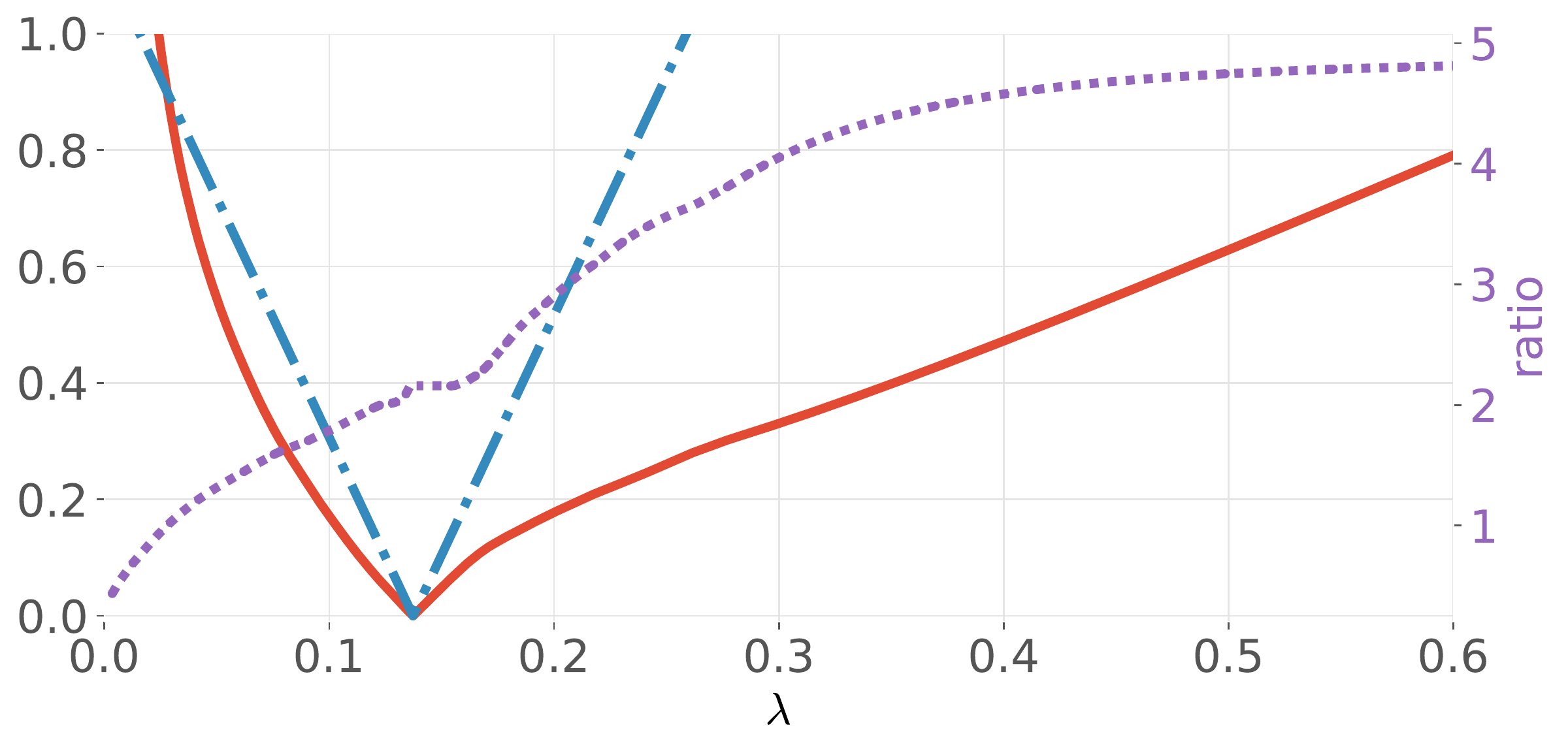}\hfill\null

  $s = 7$
  
  \null\hfill\includegraphics[width=.23\textwidth]{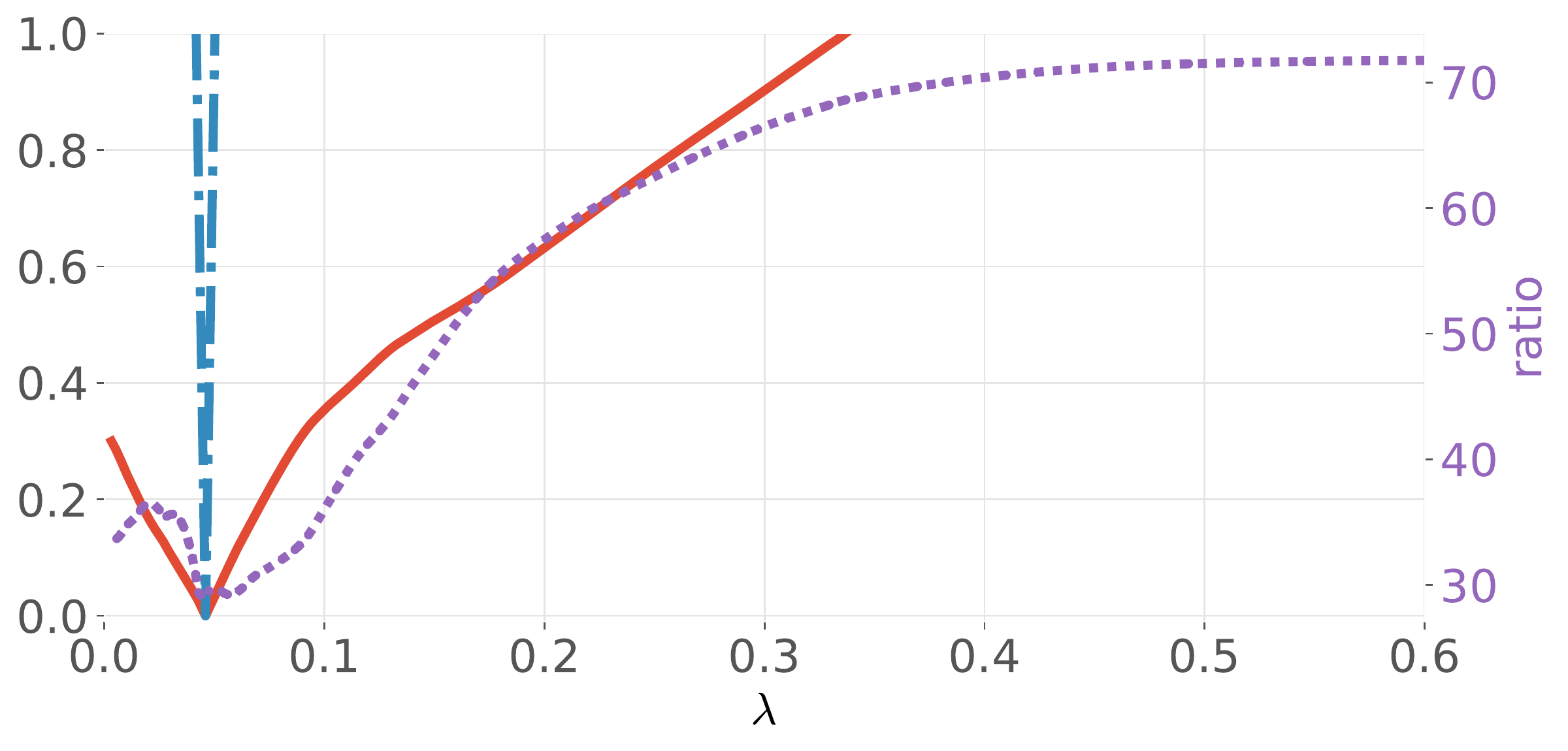}
  \hfill\includegraphics[width=.23\textwidth]{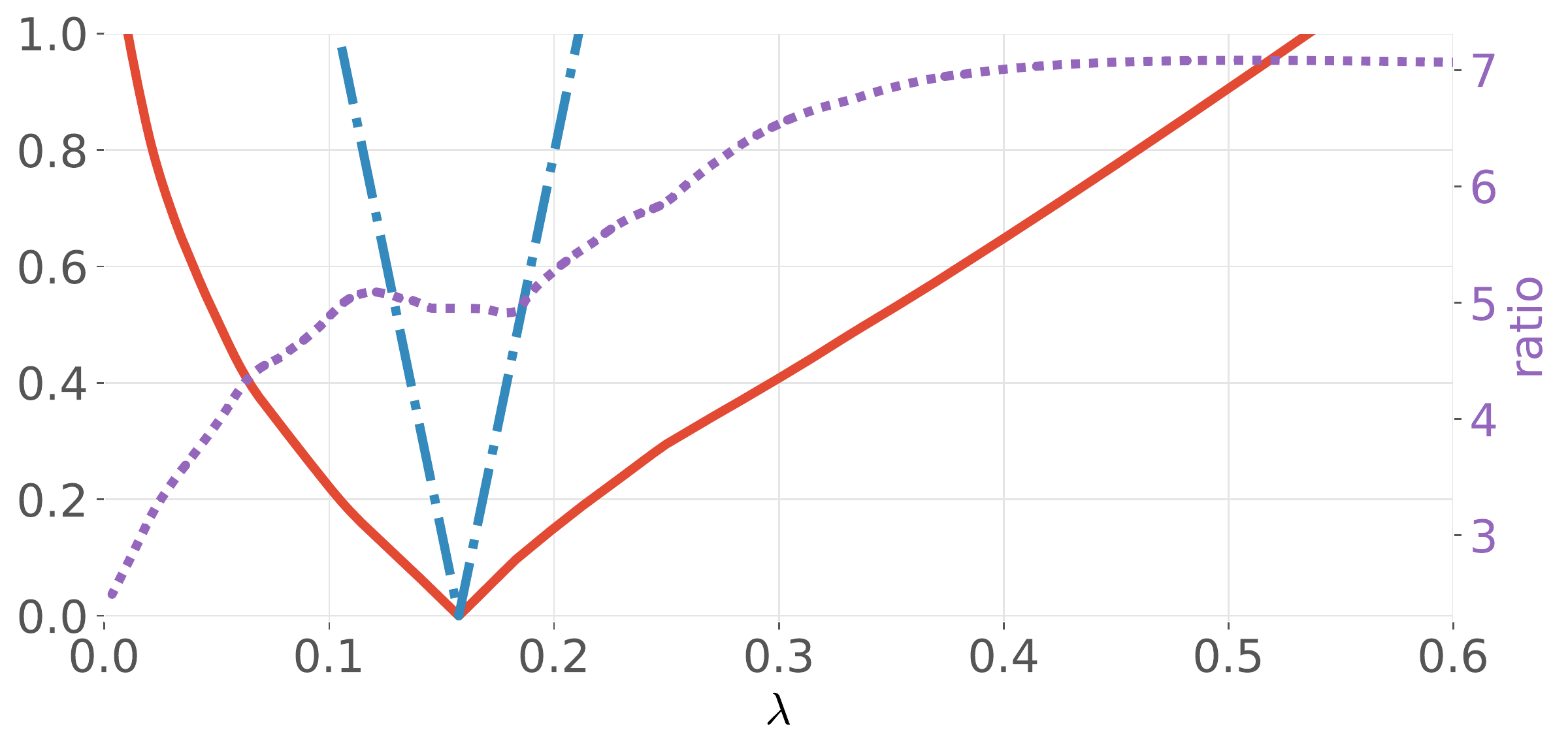}
  \hfill\includegraphics[width=.23\textwidth]{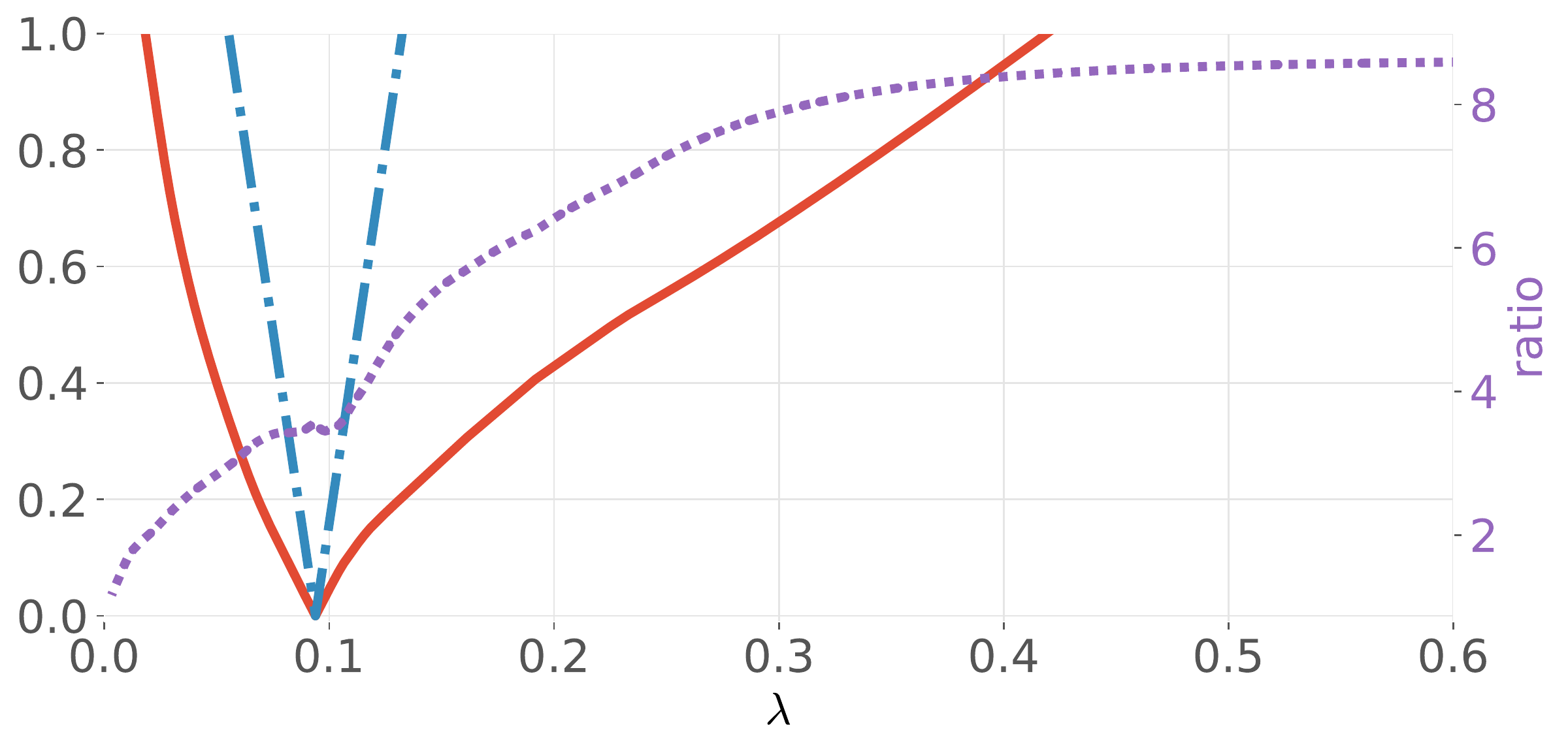}
  \hfill\includegraphics[width=.23\textwidth]{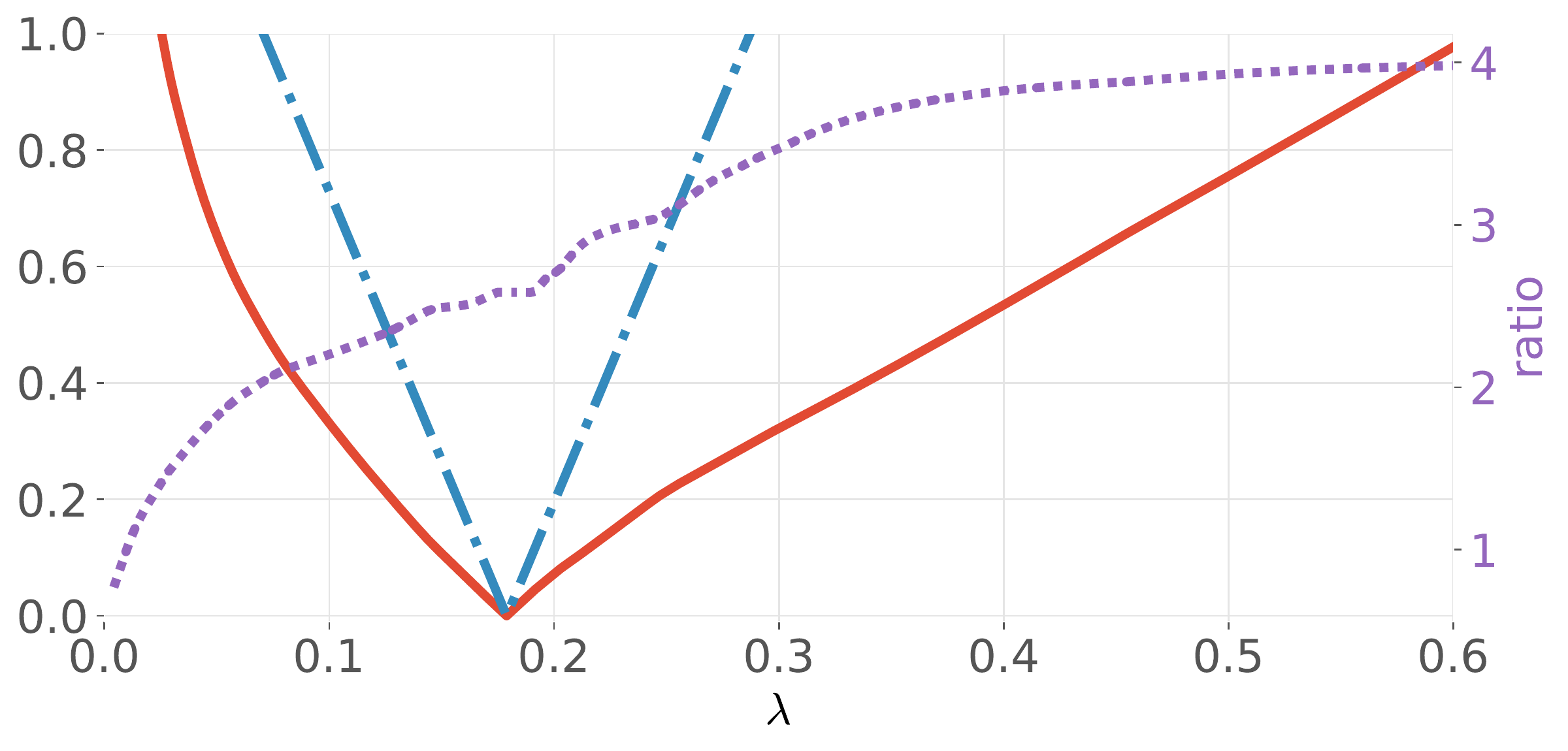}\hfill\null

  $s=15$
  
  \null\hfill\includegraphics[width=.23\textwidth]{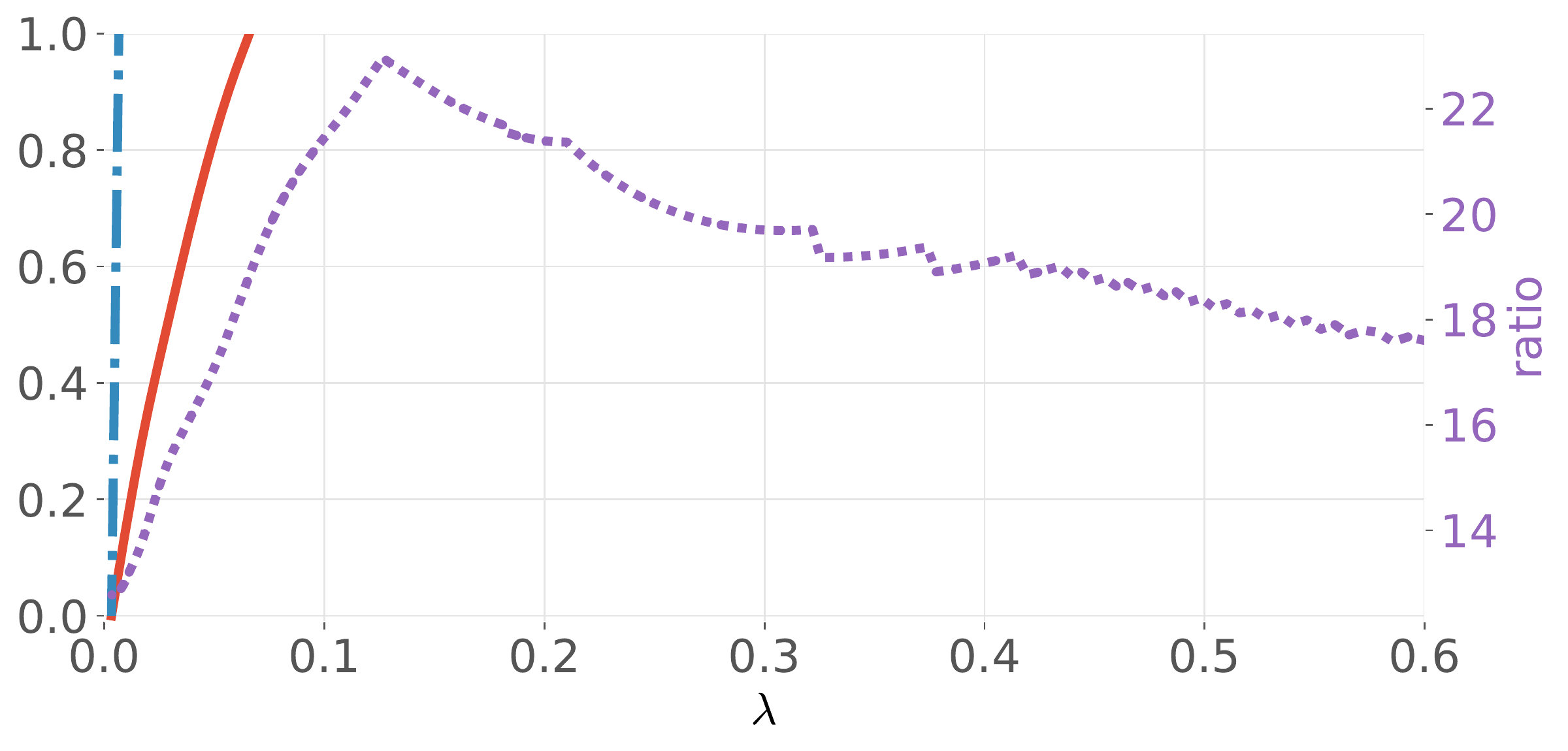}
  \hfill\includegraphics[width=.23\textwidth]{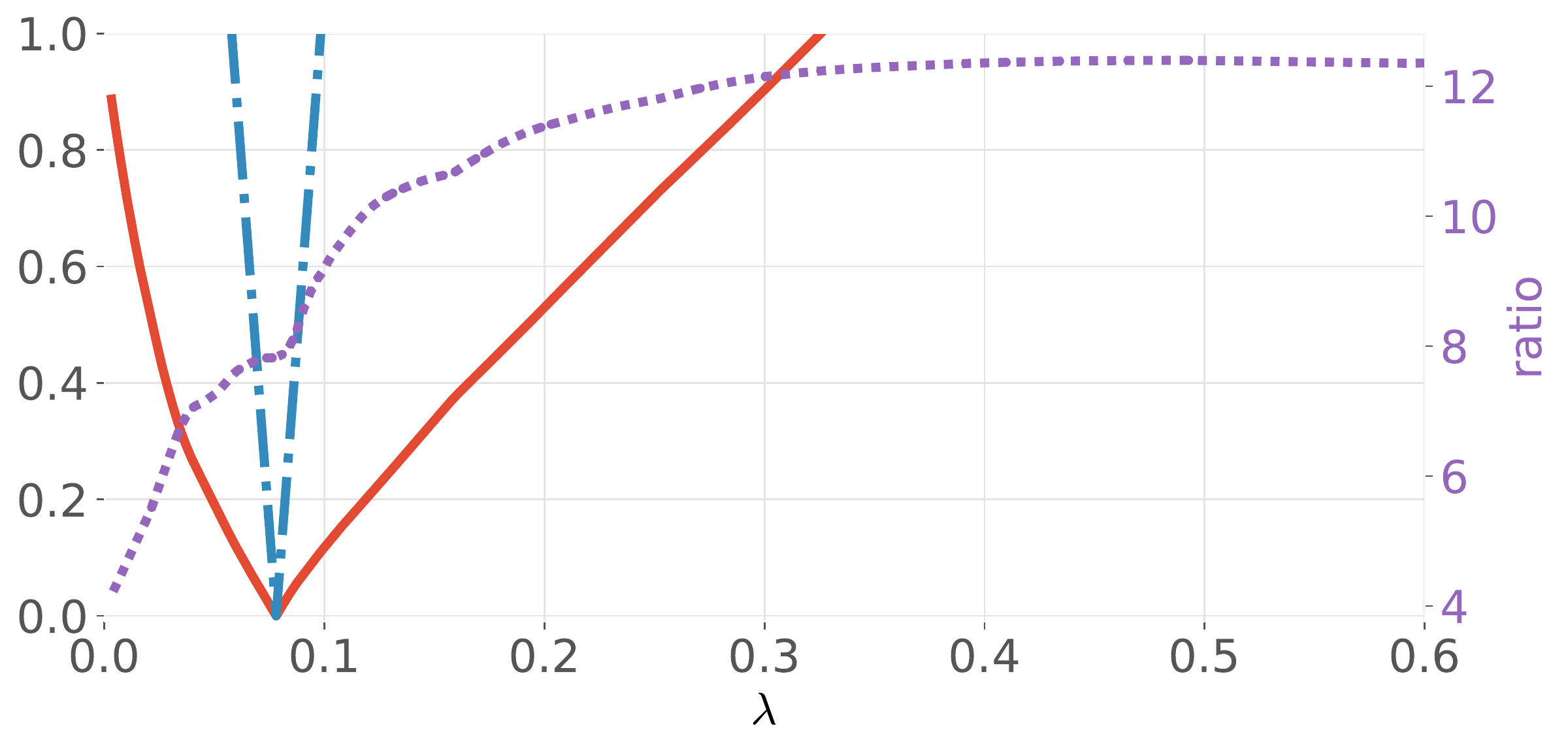}
  \hfill\includegraphics[width=.23\textwidth]{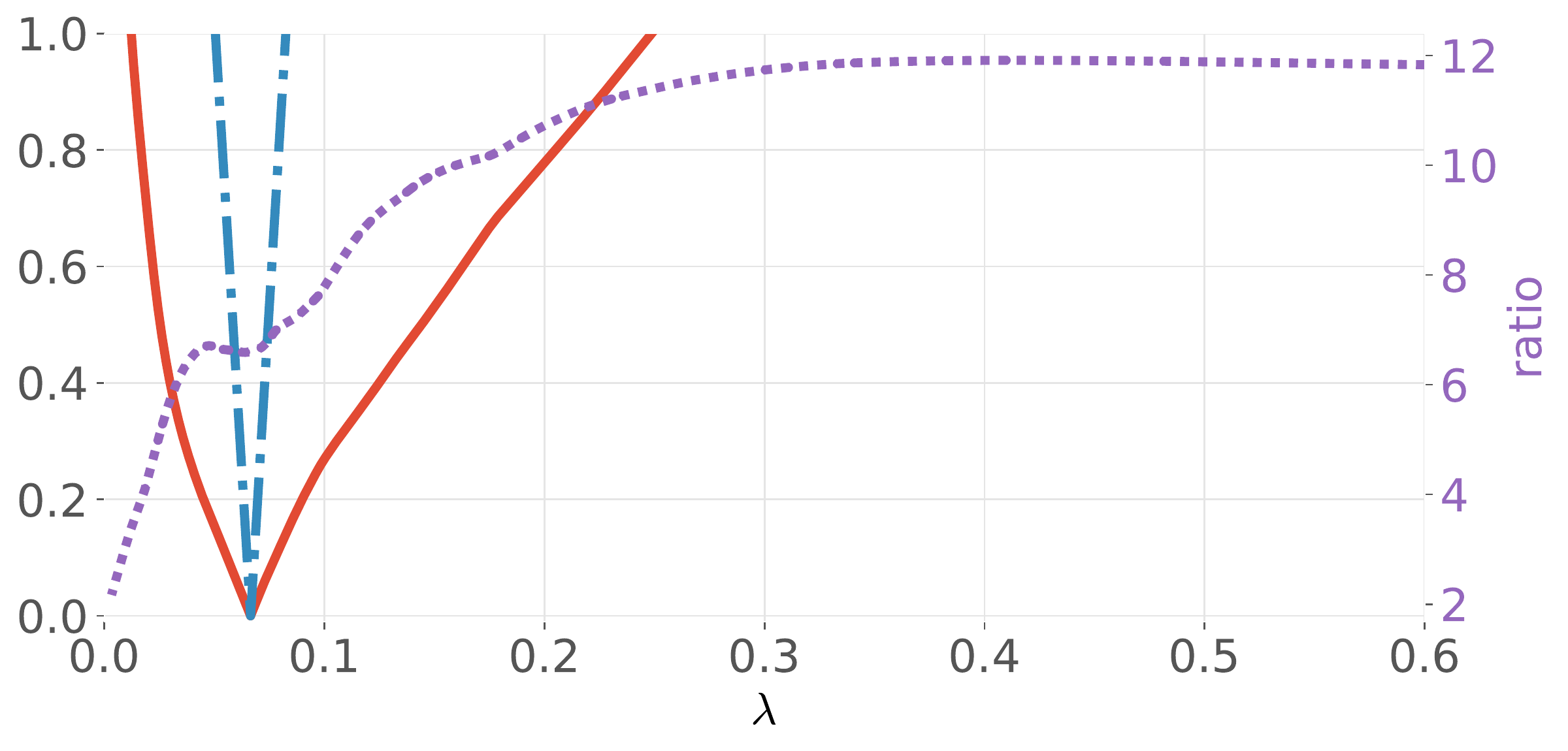}
  \hfill\includegraphics[width=.23\textwidth]{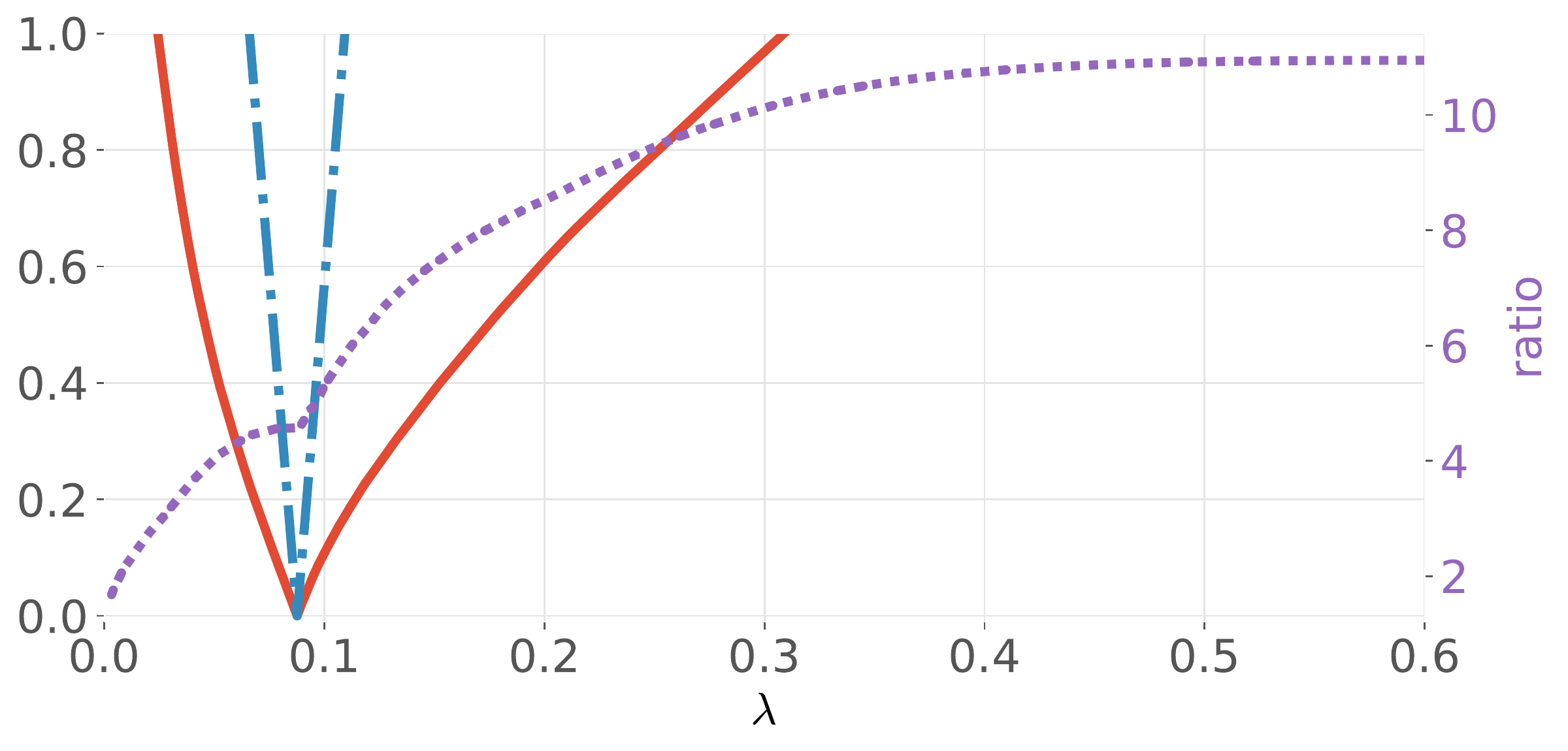}\hfill\null

  \caption{Lipschitzness of the solution mapping for $\lambda$ about
    $\bar \lambda := \lambda^{*}$ as defined in \cref{eq:lambda-star}. The red curve plots
    $\|\bar{x}(\lambda) - \bar{x}(\bar \lambda)\|$; the blue curve,
    $|\lambda - \bar \lambda| \cdot \sqrt{s}/\sigma_{\min}(A_{I})^{2}$. The
    ratio of the two is given by the purple curve, whose $y$-axis is on the
    right side of each plot. From left to right, each column corresponds to
    $m = 50, 100, 150, 200$, respectively. See \autoref{tab:numerics} for
    parameter settings and variable values. }
  \label{fig:numerics}
\end{figure}

The synthetic experiments were conducted for $s = 3, 7, 15$ (corresponding to
each row in the figure, respectively, top-to-bottom) and $m = 50, 100, 150, 200$ (corresponding
to each column, respectively, left-to-right) with $N = 200$. We selected
$\gamma = 0.1$ and for each $j \in \supp(x_{0}) := \{1, \ldots, s\}$, $(x_{0})_{j} = m + \sqrt m W_{j}$ where
$W_{j} \iid \mathcal{N}(0, 1)$. Note that for each choice of $(s, m)$, $\lambda^*$ was chosen as the empirically best choice of tuning parameter from a logarithmically spaced grid of 501 $\lambda$ values approximately centered about a nearly asymptotically order-optimal choice $\gamma \sqrt{2\log n}$. The LASSO program was solved in Python using \texttt{lasso\_path} from scikit-learn~\cite{sklearn}. 

In~\cref{tab:numerics} we report parameter values and relevant quantities
associated with~\cref{fig:numerics}. In particular, the upper bound on the
Lipschitz constant is given by
$L := \sqrt{s_{\lambda^{*}}} \cdot \sigma_{\min}^{-2}(A_{I})$. In essence, the quantity in the penultimate column determines whether~\cref{ass:Strong} holds (\nb{} $A$ is a Gaussian
random matrix, so $A_{I}$ has full rank almost surely when $|I| \leq m$). For all values of $s$
and $m$ in the experiment, $\|A_{I^{C}} (\bar b - A_{I}\bar x_{I})\|_{\infty} < \lambda^{*}$ (though we only present data up to three significant digits in~\cref{tab:numerics}). The only violation of \Cref{ass:Strong} was for $(s, m) = (15, 50)$, for which $s_{\lambda^*} = 63 > m$ meaning $A_I$ was not of  full column rank, violating condition (i). Successful recovery failed in this case, as discussed below.
\begin{table}[ht]
  \centering
  \begin{tabular}{rrrrrrcll}
    \toprule
    $s$ & $m$ & $N$ & $\eta$ & $s_{\lambda^{*}}$ & $L$\hphantom{00} & $\sigma_{\min}(A_{I})$ & $\|A_{I^{C}} (\bar b - A_{I}\bar x_{I})\|_{\infty}$ & $\lambda^{*}$ \\
    \midrule
    3 &  50 & 200 &  0.1 &  25 & 43.5\hphantom{0}    & 0.339            &  0.0831 & 0.0866 \\
    3 & 100 & 200 &  0.1 &  22 & 15.4\hphantom{0}    & 0.552            &   0.162 & 0.167 \\
    3 & 150 & 200 &  0.1 &  20 & 10.2\hphantom{0}    & 0.661            &   0.175 & 0.175 \\
    3 & 200 & 200 &  0.1 &  20 &  8.23               & 0.737            &   0.136 & 0.137 \\
    \midrule
    7 &  50 & 200 &  0.1 &  41 &  236\hphantom{.0}   & 0.165            &  0.0458 & 0.046 \\
    7 & 100 & 200 &  0.1 &  26 &   18.7\hphantom{0}  & 0.522            &  0.155 & 0.158 \\
    7 & 150 & 200 &  0.1 &  43 &   25.9\hphantom{0}  & 0.503            &  0.0931 & 0.0939 \\
    7 & 200 & 200 &  0.1 &  23 &    9.24             & 0.72\hphantom{0} &  0.175 & 0.179 \\
    \midrule
    15 &  50 & 200 &  0.1 & 63 &  301\hphantom{.0}   & 0.162            & 0.00322 & 0.00326 \\
    15 & 100 & 200 &  0.1 & 47 &   49.4              & 0.372            & 0.078   & 0.0781 \\
    15 & 150 & 200 &  0.1 & 66 &   62.5              & 0.361            & 0.0655  & 0.0665 \\
    15 & 200 & 200 &  0.1 & 67 &   46.2              & 0.421            & 0.085   & 0.0876 \\
    \bottomrule
  \end{tabular}
  
  \caption{Parameter settings corresponding with~\autoref{fig:numerics}. The
    assumption \newline
    $\|A_{I^{C}} (\bar b - A_{I}\bar x_{I})\|_{\infty} < \lambda^{*}$ is
    satisfied for all entries where $\lambda^*$ is defined as in
    \cref{eq:lambda-star} (\nb{} the two entries in the third row only
    \emph{appear} equal due to rounding error).}
  \label{tab:numerics}
\end{table}

In every panel of the figure, the dash-dot curve is indeed a strict (local) upper bound on $\|\bar x(\lambda) - \bar x (\lambda^*)\|$, supporting our theory. As predicted by the theory, the error $\|\bar x(\lambda) - \bar x (\lambda^*)\|$ is more pronounced when the problem is \emph{under-regularized} ($\lambda < \lambda^*$), and when the support size grows. Note that the $(s, m) = (15, 50)$ plot (lower-left panel in the figure)
corresponds to an unsuccessful approximate recovery of the ground truth signal
$x_{0}$, due to the relatively large sparsity as compared to the number of
measurements. In this setting, the small value of $\sigma_{\min}(A_{I})$ and the
apparently poor behaviour of $\|\bar x(\lambda) - \bar x(\lambda^{*})\|$ is
consistent with our theory. Similarly, the $(s, m) = (7, 50)$ plot (middle-left
panel in the figure) corresponds with a small value of $\sigma_{\min}(A_{I})$
due to the relatively small value of $m$ (and incidentally, it corresponds with
poor recovery of the ground truth). Again, the relatively poor behaviour of
$\|\bar x(\lambda) - \bar x(\lambda^{*})\|$ (as compared, say, with
$m \geq 100$) is consistent with our theory.

\section{Final remarks}\label{sec:Final}

\noindent
In this paper we studied the optimal value and the optimal solution function of the LASSO problem as a function of both the right-hand side (or vector of measurements) $b\in \R^m$ and the regularization (or tuning) parameter $\lambda>0$. Our analysis of the optimal value function is based on classical convex analysis, while the study of the optimal solution function is based on modern  variational analysis (in particular, differentiation of set-valued maps). As a by-product we established the (strong) metric regularity of the subdifferential of the objective function at a solution. The assumptions needed to perform this analysis were inspired by uniqueness results in the literature, and are shown to hold, \eg{} in the case where the right-hand side at the point in question admits an irreducible sparse representation with respect to the columns of the measurement matrix.  We then combined  these variational-analytic 
findings with random matrix-theoretic arguments to study the sensitivity of the LASSO solution with respect to the tuning parameter, providing upper bounds for the corresponding Lipschitz constant that hold with high probability for measurement matrices of subgaussian type.

Several questions arise naturally as a topic of future research. Can analogous statements to \Cref{th:LASSOSol} be proved for alternative formulations of $\ell_1$ minimization such as the constrained LASSO, quadratically-constrained basis pursuit, or the square-root LASSO? Similarly, the extension of our analysis to convex regularizers beyond the $\ell_1$-norm is an interesting open issue. In particular, can the analysis carried out here be generalized to the matrix setting with the nuclear norm in place of the $\ell_1$-norm?  This will certainly require a good handle on the graph of the subdifferential of the nuclear norm. Finally,~\cref{exa:Fuchs} suggests that the Lipschitz bound with respect to $\lambda$ provided by~\cref{cor:Lipschitz_lambda} might be hard to improve in general. Regarding applications to compressed sensing, another open problem is understanding whether our analysis could lead to results pertaining to the robustness to noise in the measurements. 

\section*{Acknowledgements} The authors would like to thank two anonymous referees for constructive comments that considerably improved the quality of the paper. Moreover, they are grateful to Ben Adcock for providing feedback on an earlier version of the manuscript and to Quentin Bertrand and Tony Silveti-Falls for interesting discussions.

\appendix
\section{\texorpdfstring{Supplement to~{\Cref{sec:Solution}}}{Supplement to Section 4}}
\label{SM:Sol}

In this section provide  some skipped proofs and further details on the results for the solution map of
{LASSO}. 

We commence with a fact that was employed to conclude the proof of
\Cref{lem:ConSupp}. It was implicitly used that, given $z=(x,\|x\|_1)$, then
there is a one-to-one correspondence between the \emph{active set}
$\mathcal A(z)$ for $z\in \Omega:=\epi\|\cdot\|_1$ and $\supp(x)$. To this end,
we realize that we can write the polyhedron $\Omega$ in the form
\[
  \Omega=\set{(x,\alpha)\in \R^{n+1}}{\binom{a^\nu}{-1}^T\binom{x}{\alpha}\leq 0\; (\nu\in \Lambda)}
\]
for some finite index set $\Lambda$, such that the vectors $a^\nu$ are distinct and $a^\nu_i\in\{\pm 1\}$. Then $\nu$ is (by definition) an active index at $z$ if and only if $(a^\nu)^Tx=\|x\|_1$ which is the case if and only if \[
  a_{\supp(x)}^\nu=\sgn(x_{\supp(x)}).
\]
This shows the desired correspondence.

Next, we provide detailed calculations for~\cref{rmk:A-as-parameter}, in which
$A$ is treated as a parameter of the solution map.

\begin{remark}[The matrix $A$ as a parameter]
  \label{rmk:4-10-supp}
  \cref{th:LASSOSol} can be extended to the case where the solution is
  considered as a function of $(A,b, \lambda)$ by using analogous
  arguments. First, we consider the extended function
  $f : (\mathbb{R}^{m\times n} \times \mathbb{R}^m \times \mathbb{R}_{++})
  \times \mathbb{R}^n \to \mathbb{R}^n$ defined by
  $f(A,b,\lambda,x)=\frac{1}{\lambda}\left(A^T(Ax-b)\right)$, where, in this
  case, $\mathbb{R}^{m \times n}$ is equipped with the Frobenius norm
  $\|\cdot\|_F$. Similar to the proof of \cref{th:LASSOSol}, we write
  $f((A,b,\lambda),x) = f(A,b,\lambda,x)$ and make an analogous abuse of
  notation for functions depending on both the parameters $(A,b,\lambda)$ and
  the variable $x$. First, a direct computation (\eg{}
  see~\cite{petersen2008matrix}) shows that
  \begin{align*}
    D_Af(\bar{A},\bar{b},\bar{\lambda},\bar{x}) H %
    = \bar{\lambda}^{-1}((\bar{A}^T H + H^T \bar{A})\bar{x} - H^T \bar{b}),
  \end{align*}
  for any $H \in \mathbb{R}^{m \times n}$. We now show how to generalize
  the proof of \cref{th:LASSOSol}~(a) by highlighting where its arguments need
  some modifications. The proof of the local Lipschitz continuity of the
  solution map
  $ S:(A, b ,\lambda) \mapsto \argmin_{x\in\mathbb{R}^n} \frac12 \|Ax-b\|^2 +
  \lambda \|x\|_1 $ is identical. Defining
  $G(A,b,\lambda,x) = f(A,b,\lambda,x) + F(x)$ and applying
  \cref{prop:Implicit} and \cref{lem:Sum}, we see that
  $DS(\bar A, \bar b, \bar \lambda, \bar x)$ is single-valued and locally
  Lipschitz with
  \begin{align*}
    DS(\bar A, \bar b,\bar \lambda)(H, q,\alpha)=\set{w\in \R^n}{0\in DG(\bar A, \bar b,\bar \lambda,\bar x|0)(H, q,\alpha,w)},
  \end{align*}
  and where, thanks to \cref{lem:Sum},
  \begin{align*}
    DG(\bar A, \bar b,\bar \lambda,\bar x| 0)(H, q,\alpha,w)
    & = \frac{1}{\bar \lambda}\left(\xi -{\bar A}^T\left(q+\frac{\alpha}{\bar \lambda}(\bar A\bar x-\bar b)-\bar Aw\right)\right)\\
    & + D(\p\|\cdot\|_1)(\bar x|-{\bar A}^T(\bar A\bar x-\bar b))(w),
  \end{align*}
  where $\xi %
  = \xi(\bar A, \bar b, \bar x, H) := (\bar A^T H + H^T \bar A)\bar x - H^T \bar b \in \mathbb{R}^n$. Arguing as in the proof of \cref{th:LASSOSol}, we obtain
  \begin{align*}
    0 &\in DG(\bar A, \bar b,\bar \lambda,\bar x|0)(H, q,\alpha,w)
    \\
      & \implies %
        \begin{cases}
          w_{I^C_<}=0
          \\
          \bar A_I^T\left(q+\frac{\alpha}{\bar \lambda}(\bar A\bar x-\bar b)-\bar Aw\right) + \xi_I= 0
          \\
          w_i\left(\bar A_i^T\left(q+\frac{\alpha}{\bar \lambda}(\bar A\bar x-\bar b)-\bar Aw\right) + \xi_i\right)= 0, \; \forall i\in I_{=}^C.
        \end{cases}
  \end{align*}
  This, in turn, leads to the formula
  \begin{align*}
    S'(\bar A, \bar b, \bar \lambda)(H,q,\alpha) 
    = L_K\left(({\bar A}_K^T \bar A_K)^{-1} %
    \left(%
    {\bar A}_K^T q +\frac{\alpha}{\bar \lambda} {\bar A}^T_K(\bar A \bar x - \bar b) + \xi_K\right)
    \right). 
  \end{align*}
  Arguing again as in the proof of \cref{th:LASSOSol}, observing that the
  function $\sigma_{\max}(\cdot)$ is continuous, and choosing a sequence
  $(A_k,b_k,\lambda_k) \to (A,b,\lambda)$ such that $(A_k)_K$ has full column
  rank for every $k\in \mathbb{N}$ (by passing to a subsequence if
  necessary), we see that there exists $(\bar H, \bar q, \bar \lambda) \in \mathbb{B}$
  such that
  \begin{align*}
    L
    & \leq \frac{1}{\sigma_{\min}(\bar A_K)^2} \left\|\bar A_K^T \bar q+\frac{\bar \alpha}{\bar \lambda}\bar A_K^T(\bar A\bar x-\bar b) - (\bar A_K^T \bar H + \bar H_K^T \bar A)\bar x + \bar H_K^T \bar b\right\|\\
    & \leq \frac{1}{\sigma_{\min}(\bar A_K)^2} \left( \sigma_{\max}(\bar A_K) + \left\|\frac{\bar A_J^T(\bar A\bar x-\bar b)}{\bar \lambda}\right\| + \max_{\|H\|_F \leq 1} \left\|\bar A_K^T H \bar x + H_K^T (\bar A\bar x - \bar b)\right\|\right)\\
    & \leq \frac{1}{\sigma_{\min}(\bar A_K)^2} \left( \sigma_{\max}(\bar A_K) + \left\|\frac{\bar A_J^T(\bar A\bar x-\bar b)}{\bar \lambda}\right\| + \sigma_{\max}(\bar A_K)  \|\bar x\|+ \|\bar A\bar x - \bar b\|\right)\\
    & \leq \frac{1}{\sigma_{\min}(\bar A_J)^2} \left( \sigma_{\max}(\bar A_J) (1 + \|\bar x\|) + 
      \left\|\frac{\bar A_J^T(\bar A\bar x-\bar b)}{\bar \lambda}\right\| + \|\bar A\bar x - \bar b\|\right)
  \end{align*}
  The extension of \cref{th:LASSOSol}~(b) follows by analogous arguments, with
  $J = I$.\hfill$\diamond$
\end{remark}

\section{\texorpdfstring{Supplement to~{\Cref{sec:Appli}}}{Supplement to Section 5}}

We include here technical background from high-dimensional probability, which is
required to
prove~\cref{prop:applic_variational_lambda,prop:applic_variational_lambda_no_sparsity}. We
state the result~\cref{lem:submat-least-singvals} in a form similar
to~\cite[Theorem~9.1.1]{Ver_17} and moreover using the optimal dependence on
$\beta$ established in~\cite{Jeo_20}. For integers $1 \leq s \leq n$,
recall the definition of $\Sigma_{s}^{n} := \{x \in \reals^{n} : \|x\|_{0} \leq s\}$. In addition, let $T_{n,s}:= \sqrt s \mathbb{B}_{1} \cap \mathbb{B}_{2}$ and
$J_{n, s} := \co (\Sigma_{s}^{n} \cap \mathbb{B}_{2})$. Throughout we
let $\mathcal{C} \in \{ J_{n,s}, T_{n,s}\}$.

\begin{definition}[Restricted Isometry Property]
  \label{def:vrip}
  Let $0 < \alpha_{-} \leq 1 \leq \alpha_{+} < \infty$ and $s \in \nats$. A
  matrix $A \in \reals^{m \times n}$ satisfies the Restricted Isometry Property
  (RIP) of order $s$ with parameters $(\alpha_{-}, \alpha_{+})$ if
  \begin{align*}
    \alpha_{-} \| x \| \leq \| Ax \| \leq \alpha_{+} \| x \|,
    \qquad\forall x \in \Sigma_{s}^{n}. 
  \end{align*}
\end{definition}

\subsection{Matrix deviation inequalities}
\label{sec:concentration-inequalities}


Next we introduce two auxiliary results that will be useful in
establishing our results. The first is~\cite[Corollary~1.2]{Jeo_20}. 

\begin{lemma}[{\cite[Corollary~1.2]{Jeo_20}}]
  \label{lem:jeo20coro12}
  Let $A \in \reals^{m \times n}$ be a $\beta$-subgaussian matrix and let $T \subseteq \reals^{n}$ be a bounded set. Then
  \begin{align*}
    \E \sup_{x \in T} \left| \| Ax \| - \sqrt m \| x \| \right| \leq C \beta \sqrt{\log \beta} \left[ \gmw(T) + \rad(T) \right],
  \end{align*}
  and for any $u \geq 0$, with probability at least $1 - 3 e^{-u^{2}}$,
  \begin{align*}
    \sup_{x \in T} \left|  \| Ax \| - \sqrt m \| x \| \right| \leq C\beta \sqrt{\log \beta} \left[ \gmw(T) + u \cdot \rad(T) \right]. 
  \end{align*}
  Here, $C > 0$ is an absolute constant. 
\end{lemma}
Above, $\rad T := \sup_{x \in T} \|x\|$ and $\gmw(T)$ is the Gaussian mean width of a set $T \subseteq \reals^{n}$:
\begin{align*}
  \gmw(T) := \E \sup_{x \in T} \ip{x}{g}, \qquad \text{where} \; g_i \iid \mathcal{N}(0, 1), \quad i =1,\ldots,n. 
\end{align*}
Observe that $\gmw(T) = \gmw(\conv T)$, where $\conv T$ denotes the convex hull
of $T$. We now summarize~\cite[Exercises~10.3.8--9]{Ver_17} in the
following lemma.

\begin{lemma}[{\cite[Exercises~10.3.8--9]{Ver_17}}]
  \label{lem:ver17ex10389}
  For $C \geq c > 0$ being absolute constants,
  \begin{align*}
    c \sqrt{s \log(2n/s)} \leq \gmw(J_{n,s}) \leq \gmw(T_{n,s}) \leq 2\gmw(J_{n,s}) \leq C \sqrt{s \log(en/s)}. 
  \end{align*}
\end{lemma}
The following lemma specializes~\cite[Corollary~1.2]{Jeo_20}
using~\cite[Exercises~10.3.8--9]{Ver_17}.
\begin{lemma}[Sparse subgaussian deviations]
  \label{lem:sparse-subgaussian-deviations}
  Let $A \in \reals^{m \times n}$ be a $\beta$-subgaussian matrix and let
  $\mathcal{C}$ be either $J_{n,s}$ or $T_{n,s}$. For an appropriate choice of
  absolute constant $C >0$, it holds that
  \begin{align*}
    \E \sup_{x \in \mathcal{C}} \left| \| Ax \| - \sqrt m \| x \| \right| \leq C \beta \sqrt{\log \beta} \left[ \sqrt{s \log (e n/s)} + 1 \right]
  \end{align*}
  and for any $u \geq 0$, with probability at least $1 - 3e^{-u^{2}}$,
  \begin{align*}
    \sup_{x \in \mathcal{C}} \left| \| Ax \| - \sqrt m \| x \| \right| \leq C \beta \sqrt{\log \beta} \left[ \sqrt{s \log (e n /s)} + u \right].
  \end{align*}
\end{lemma}
\begin{proof}[{Proof of~\cref{lem:sparse-subgaussian-deviations}}]
  We demonstrate the proof for the expectation expression, since the
  probability bound follows by an identical set of
  steps. Combine~\cref{lem:jeo20coro12} and~\cref{lem:ver17ex10389} to obtain
  \begin{align*}
    \E \sup_{x \in \mathcal{C}} \left| \| Ax \| - \sqrt m \| x \| \right| \leq C_{1} \beta \sqrt{\log \beta} \left[ C_{2} \sqrt{s \log (e n/s)} + 1 \right],
  \end{align*}
  noting that $\rad \mathcal{C} = \sup_{x \in \mathcal{C}} \| x \| = 1$. Above,
  we write $C_{1}$ to denote the absolute constant
  from~\cite[Corollary~1.2]{Jeo_20} and $C_{2}$ the one
  from~\cite[Exercise~10.3.8]{Ver_17}. We obtain
  \begin{align*}
    C_{1} \beta \sqrt{\log \beta} \left[ C_{2} \sqrt{s \log (e n/s)} + 1 \right] %
    \leq C \beta \sqrt{\log \beta} \left[ \sqrt{s \log (e n/s)} + 1 \right],
  \end{align*}
  where $C := C_{1} \max \{C_{2}, 1\}$ is the absolute constant appearing in the
  result statement, completing the proof.
\end{proof}
\subsection{Bounding singular values of submatrices}
\label{sec:bounds-singvals}

Finally, we present a probabilistic bound on the singular values of certain
families of $s$-column submatrices, which in turn establishes a restricted
isometry result for the class of matrices considered in~\cref{sec:appli-lasso1}.

\begin{lemma}
  \label{lem:submat-least-singvals}
  Let $1 \leq s \leq m \leq n$ be integers and let $\delta \in (0, 1)$. Suppose
  that
  \begin{align*}
    m %
    \geq C \delta^{-2} \beta^{2}\log \beta \left[ s \log(en/s) %
    + \log(3 / \varepsilon) \right],
  \end{align*}
  where $C > 0$ is an absolute constant. If
  $A \in \reals^{m \times n}$ is a normalized $\beta$-subgaussian matrix then
  with probability at least $1 - \varepsilon$ on the realization of $A$ it holds
  that
  \begin{align*}
    1 - \delta \leq \min_{|K| \leq s} \sigma_{\min}(A_{K}) \leq \max_{|K| \leq s} \sigma_{\max}(A_{K}) \leq 1 + \delta. 
  \end{align*}
  In particular, $A$ satisfies RIP of order $s$ with parameters $(1-\delta, 1 + \delta)$. 
\end{lemma}
\begin{proof}[{Proof of~\cref{lem:submat-least-singvals}}]
  Let $\tilde C > 0$ be the absolute constant
  from~\cref{lem:sparse-subgaussian-deviations} and set $C :=
  2\tilde{C}^{2}$. Recall that in~\cref{lem:sparse-subgaussian-deviations}, $\mathcal{C}$ could
  have been either $J_{n,s}$ or $T_{n,s}$. Choose $\mathcal{C} := J_{n,s}$. For $\delta$ as given in the result statement, if
  $m \geq 2\tilde{C}^{2} \delta^{-2} \beta^{2}\log \beta \left[ s \log(en/s) +
    \log(3 / \varepsilon) \right]$ and $A \in \reals^{m \times n}$ is a
  normalized $\beta$-subgaussian matrix (\ie there exists a $\beta$-subgaussian
  matrix $\tilde A \in \reals^{m \times n}$ with $A = m^{-1/2} \tilde A$), then,
  by~\cref{lem:sparse-subgaussian-deviations} with
  $u := \sqrt{\log(3/\varepsilon)}$, it holds with probability at least
  $1 - \varepsilon$ that
  \begin{align*}
    \sup_{x \in J_{n,s}} \left| \| Ax \| - \| x \| \right| %
    & \leq \frac{\tilde{C} \beta \sqrt{\log \beta} \left[ \sqrt{s \log (en/s)}  + \sqrt{\log(3/\varepsilon)} \right]}{\sqrt m} %
    \\
    & \leq \delta \frac{ \sqrt{s \log (en/s)}  + \sqrt{\log(3/\varepsilon)}}{\sqrt{2} \cdot \sqrt{s \log (en/s) + \log(3/\varepsilon)}} %
    \\
    & \leq \delta.
  \end{align*}
  The first inequality is a consequence of the referenced result; the second
  obtained by substituting for $m$ and simplifying. The last line follows by an
  application of Jensen's inequality (\ie{} for $a, b \geq 0$ it holds that
  $\sqrt 2 \cdot \sqrt{a + b} \geq \sqrt a + \sqrt b$).

  Restricting to $x \in J_{n,s}\cap\sphn$ thereby yields
  $\sup_{x \in J_{n,s} \cap \sphn} \left| \| Ax \| - 1 \right| \leq
  \delta$. In particular,
  \begin{align*}
    1 - \delta \leq \| Ax \| \leq 1 + \delta, \qquad\forall x\in J_{n,s} \cap \sphn.
  \end{align*}
  Using positive homogeneity of $\|\cdot\|$, it follows that
  \begin{align}
    \label{eq:restricted-sing-vals}
    (1 - \delta) \| x \| \leq \| Ax \| \leq (1 + \delta) \| x \|, \qquad\forall x \in J_{n,s}. 
  \end{align}
  It
  follows that, for any $K \subseteq \{1,\ldots,n\}$ with $|K| \leq s$,
  \begin{align*}
    (1 - \delta) \| x \| \leq \| A_{K}x \| \leq (1 + \delta) \| x \|, %
    \qquad\forall x \in \reals^{|K|}. 
  \end{align*}
  Specifically, \cref{eq:restricted-sing-vals} gives the desired result in view of the definition of $J_{n,s}$:
  \begin{align*}
    & \sup_{x \in J_{n,s}\cap \sphn} \left| \| Ax \| - 1 \right| \leq \delta
    \\
    & \implies\quad \sup_{|K| \leq s}\sup_{x \in \sphn[|K|]} \left| \| A_{K} x\| - 1 \right| \leq \delta \\
    & \implies\quad (1 - \delta) \leq \| A_{K}x \| \leq (1 + \delta), \qquad\forall K \subseteq \{1,\ldots,n\} \; \text{with} \; |K| \leq s, \quad \forall x \in \sphn[|K|]
    \\
    & \implies\quad 1 - \delta \leq \min_{|K| \leq s} \sigma_{\min}(A_{K}) \leq \max_{|K| \leq s} \sigma_{\max}(A_{K}) \leq 1 + \delta. 
  \end{align*}
  (Note the final implication is obtained from one characterization of the
  extremal singular values of a matrix --- see~\cite[(4.5)]{Ver_17}.)
  Clearly $A$ satisfies RIP of order $s$ with parameters $(1-\delta, 1+\delta)$ in
  view of~\cref{def:vrip}.
\end{proof}


\end{document}